\documentclass{article}
\pdfoutput=1

\usepackage{hyperref}
\usepackage{graphics}
\usepackage{graphicx}
\usepackage{subfigure}
\usepackage{float}
\usepackage{epstopdf}
\usepackage{epsfig}
\usepackage{color}
\usepackage{amsmath}
\usepackage{amsthm}
\usepackage{amssymb,amsfonts}
\providecommand{\keywords}[1]{{\textit{Key words.}} #1}
\providecommand{\ams}[1]{{\textit{AMS subject classifications.}} #1}
\usepackage[T1]{fontenc}
\usepackage{etoolbox}

\usepackage[english]{babel}
 \newtheorem{thm}{Theorem}[section]

 \newtheorem{prop}{Proposition}[section]
\newcommand{\R}{\mathbb R}
\newcommand{\ppf}[2]{\frac{\partial #1}{\partial #2}}

\newcommand{\pd}[2]{\partial_{#1}^{#2}}

\newcommand{\TheTitle}{A space-time finite element method for neural field equations with transmission delays}

\title{{\TheTitle}}%\thanks{This work was funded by the Fog Research Institute under contract no.~FRI-454.}}

\author{
  M\'onika~Polner\thanks{Bolyai Institute, University of Szeged, H-6720 Szeged, Aradi v\'ertan\'uk tere 1, Hungary and ELI-ALPS, ELI-HU Ltd, Dugonics t\' er 13, Szeged 6720, Hungary. (\href{mailto:polner@math.u-szeged.hu}{polner@math.u-szeged.hu}).}
  \and
  J.~J.~W.~van~der~Vegt\thanks{Mathematics of Computational Science Group, Department of Applied Mathematics, University of Twente, P.O. Box 217, 7500 AE, Enschede, The Netherlands. (\href{mailto:j.j.w.vandervegt@utwente.nl}{j.j.w.vandervegt@utwente.nl}).}
  \and
  S.~A.~van~Gils\thanks{Applied Analysis, Department of Applied Mathematics, University of Twente, P.O. Box 217, 7500 AE, Enschede, The Netherlands. (\href{mailto:s.a.vangils@utwente.nl}{s.a.vangils@utwente.nl})}
}

\begin{document}

\maketitle

\begin{abstract}
We present and analyze a new space-time finite element method for the 
solution of neural field equations 
with transmission delays. The numerical treatment of these systems is rare in the literature and 
currently has several restrictions on the spatial domain and the functions involved, 
such as connectivity and delay functions. The use of a space-time discretization, 
with basis functions that are discontinuous in time and continuous in space (dGcG-FEM), is a 
natural way to deal with space-dependent delays, which is important for many neural field applications. 
In this article we provide a detailed description of a space-time dGcG-FEM  algorithm for neural delay 
equations, including an a-priori error analysis. We demonstrate the application of the dGcG-FEM algorithm 
on several neural field models, including problems with an inhomogeneous kernel.

\vspace*{.3cm}
% REQUIRED
\noindent\keywords{Neural fields, transmission delays, discontinuous Galerkin, finite element methods, space-time methods}

\vspace*{.3cm}
% REQUIRED
\noindent\ams{65M60, 65M15, 65R20, 37M05, 92C20}

\end{abstract}
\section{Introduction}
The motivation of this work is the need for numerical methods that can accurately and efficiently 
discretize delayed integro-differential equations originating from neural field models, 
in particular when the delay in the system is space dependent. Only a few studies 
considered so far the numerical treatment of neural field systems, 
see \cite{faugeras}, \cite{hutt}, \cite{evelyn} and references therein. In \cite{faugeras}, the authors used special 
types of delay and connectivity functions in order to reduce the spatial discretization to a large system 
of delay differential equations with constant time delays. 
This system was then solved with the Matlab solver {\it dde23.} In \cite{hutt} a new numerical scheme 
was introduced that includes a convolution structure and hence allows the implementation of fast 
numerical algorithms. In both studies the connectivity kernel depends on the distance between two 
spatial locations. This choice has been shown to model successfully neural activity known from 
experiments, it introduces, however, also a limitation to the applicability of the presented techniques. 

Here we propose the use of space-time finite element methods using discontinuous basis functions in time 
and continuous basis functions in space (dGcG-FEM), which are 
well established to solve ordinary and partial differential equations, e.g. \cite{eriksson_1}, 
\cite{long_time}, \cite{eriksson_johnson_thomee}, \cite{hughes}, \cite{thomee_book}, \cite{vandervegt2002}. 
The novelty of this work is the successful application of the space-time dGcG-method to the neural field equations. 
The motivation of our choice is that the time-discontinuous Galerkin method has good long-time accuracy, \cite{long_time}, 
\cite{thomee_book}. Moreover, the use of a space-time discretization is a 
natural way to deal also with the space-dependent delays. As it will be discussed later, there is 
no need in a space-time method to interpolate the solution from previous time levels. 
The space-time dGcG-method was successfully used for 
stiff systems and is well suited for mesh adaptation, which is of great 
importance when local changes in the solution are of interest. 
Further benefits are that we do not need to make restrictions, neither to the 
functions involved in the system, such as the connectivity kernel or the delay function, 
nor to the dimension or shape of the spatial domain. 

In this article we present a novel space-time dGcG-method for delay differential equations. 
We provide a theoretical analysis of the stability and order of accuracy of the numerical discretization 
and demonstrate its application on a number of neural field problems. 
We focus on the design and an a-priori error analysis of the space-time dGcG-FEM for nonlinear neural 
field equations with space dependent delay. 

The outline of this article is as follows. In the introductory Section \ref{s:intro} we recall a 
mathematical model for neural fields. In Section \ref{s:time stepping} we introduce the 
space-time dGcG-FEM method. The main difficulty is the treatment 
of the delay term in the neural field equations, which is investigated in detail in Section \ref{s:how to}. 
An a-priori error analysis of the space-time discretization is given in Section \ref{s:error analysis}. 
Next, we show in Section \ref{s:simulations} some numerical simulations 
for the neural field equations in one spatial dimension with one population. 
These examples are taken from literature, \cite{faugeras}, \cite{vg}, where both analytical and numerical 
results are known for comparison. We demonstrate some further computational benefits of the space-time dGcG-FEM 
by introducing an inhomogeneous kernel in the delay term in Section \ref{s:inhomogenity}. 
The numerical algorithms presented in \cite{faugeras} and \cite{hutt}, are not 
suitable for the treatment of local inhomogeneities.  

In consecutive papers we will show computations on more complicated spatial domains and extend the model 
to more populations in the neural field system.
%%%%%%%%%%%%%%%%%%%%
\section{Neural fields with space dependent delays}\label{s:intro}
The mathematical model for neural fields with space-dependent delays is as follows. 
Consider $p$ populations consisting of neurons distributed over a 
bounded, connected and open domain $\Omega\subset \mathbb{R}^d,$ $d=1,2,3.$ For each $i,$ the variable 
$V_i(t,r)$ is the membrane potential at time $t,$ averaged over those neurons in the $i$th population 
positioned at $r\in\Omega.$ These potentials are assumed to evolve according to the following system of 
integro-differential equations 
\begin{equation}\label{NF}
\frac{\partial V_i}{\partial t}(t,r)=-\alpha_i V_i(t,r)+\sum_{j=1}^{p}\int_\Omega J_{ij}(r,r',t)S_j(V_j(t-\tau_{ij}(r,r'),r'))d\, r',
\end{equation}
for $i=1,\dots, p.$ The intrinsic dynamics exhibits exponential decay to the baseline level $0,$ as 
$\alpha_i>0.$ The propagation delays $\tau_{ij}(r,r')$ measure the time it takes for a signal sent by a 
type-$j$ neuron located at position $r'$ to reach a type-$i$ neuron located at position $r.$ 
The function $J_{ij}(r,r',t)$ represents the connection 
strength between population $j$ at location $r'$ and population $i$ at location $r$ at time $t.$ 
The firing rate functions are $S_j.$ For the definition and interpretation of these functions we refer to \cite{veltz}. 
Some examples will be given in later sections. 

Throughout this paper we consider a single population, $p=1,$ in a bounded domain $\Omega\subset \R^d,$ 
on a time interval 
$[t_0,T),$ with $T>t_0$ the final time, 
\begingroup
\allowdisplaybreaks
\begin{equation}\label{NF-1pop_model}
\frac{\partial u}{\partial t}(t,x)=-\alpha u(t,x)+\int_\Omega J(x,r)S(u(t-\tau(x,r),r))d\, r,\quad \alpha>0.
\end{equation}
Note that we will only deal with autonomous systems. Therefore we assume from here on that the connectivity 
does not depend on time. We assume that the following hypotheses are satisfied for the 
functions involved in the system, (as in \cite{vg}):
the connectivity kernel $J\in C(\bar\Omega\times\bar\Omega),$ 
the firing rate function $S\in C^\infty (\R)$ and its $k$th derivative is bounded for every $k\in \mathbb{N}_0,$ 
the delay function $\tau\in  C(\bar\Omega\times\bar\Omega)$ is non-negative.

Without loss of generality, we take $t_0=0.$ From the assumption on the delay function $\tau,$ 
we may set
\[
0< \tau_{max}=\sup_{(x,r)\in\bar\Omega\times\bar\Omega} \tau(x,r)<\infty.
\]
Note that when $\tau_{max}=0,$ the delay function $\tau(x,r)=0$ for all $(x,r)\in \bar\Omega\times\bar\Omega,$ and 
in this case (\ref{NF-1pop_model}) reduces to an integro-differential equation without delay. As we will see later, our 
numerical method can handle this case as well. 

Let $Y=C(\bar\Omega)$ and set $X=C\left([-\tau_{max},0];Y\right).$ For $\varphi\in X,$ $s\in[-\tau_{max},0]$ and for 
$x\in\Omega$ we write $\varphi(s)(x)=\varphi(s,x),$ and its norm is given by
\[
\| \varphi\|_X =\sup_{s\in[-\tau_{max},0]}\| \varphi(s,\cdot)\|_Y,
\]
where $\| \varphi(s,\cdot)\|_Y=\sup_{x\in\Omega}|\varphi(s,x)|.$ 
From the assumption on the connectivity kernel, it follows that it is bounded in the following norm
\[
\Vert J\Vert_C = \sup_{(x,r)\in\bar\Omega\times\bar\Omega}|J(x,r)|.
\]
We use the traditional notation for the state of the system at time $t$ 
\[
u_t(s)=u(t+s)\in C(\bar\Omega),\quad s\in[-\tau_{max},0],\ t\geq 0.
\]
Define the nonlinear operator $G:X\to Y$ by
\begin{equation}
G(\varphi)(x)=\int_\Omega J(x,r)S\left(\varphi(-\tau(x,r),r)\right)d\, r.
\end{equation}
Then the neural field equation (\ref{NF-1pop_model}) can be written as a delay differential equation (DDE) as
\begin{equation}\label{DDE_G}
\ppf{u}{t}(t)=-\alpha u(t)+G(u_t),
\end{equation}
where the solution is an element of $C([-\tau_{max},\infty);Y)\cap C^1([0,\infty);Y).$ Similarly, we have the 
state of the solution at time $t$ defined as $u_t(s)(x)=u(t+s,x),$ 
$s\in[-\tau_{max},0],$ $t\geq 0,$ $x\in\Omega.$ It was shown in \cite{vg} that under 
the above assumptions on the connectivity, the firing rate function and delay, the operator $G$ is well-defined and it satisfies a global Lipschitz condition.  

Note that the assumptions on the firing rate function $S$ imposed in \cite{vg} were needed for further 
analysis of the neural field equations. For the numerical analysis presented in this paper it is 
sufficient to assume that $S$ is Lipschitz continuous.

%%%%%%%%%%%%%%%%%%%%%%%%%%%%%%%%%%%%%%%%%%%%%
\section{The discontinuous Galerkin finite element \\ method}\label{s:time stepping}
The starting point of our numerical discretization is the weak formulation. 
The numerical method is investigated for the nonlinear equation \eqref{DDE_G}, which 
may be written in variational form as: Find $u\in C^1\left([0,T),Y\right)\cap C\left([-\tau_{max},T),Y\right)$ such that
\begin{align}\label{weak_formC} 
&\Bigl(\ppf{u}{t}(t)+\alpha u(t),v\Bigr)-\left(G(u_t),v\right)=0, &\forall v\in Y,\ \forall t\in(0,T),\\[7pt]
&\ u(s)=u_0(s), &s\in[-\tau_{max},0],\label{initial_condition}
\end{align}
where $(\cdot,\cdot)$ is the usual $L^2(\Omega)$ inner product. Here the delay contribution is 
expressed as 
\[
\bigl(G(u_t),v\bigr)=\int_\Omega G(u_t)(x)v(x) dx
=\int_\Omega\int_\Omega J(x,r)S\left(u_t(-\tau(x,r),r)\right)dr\,v(x)dx.
\]
Note that for any $t>0,$ all functions 
in the inner product are elements of $Y=C(\bar\Omega),$ which is a dense subset of $L^2(\Omega),$ 
hence the inner product is well-defined.  
%%%%%%%%%%%%%%%%%%%%%%%%%%%%%%%%%%%%%%%%%%%%%%%%%%%%%%%%
\subsection{The space-time dGcG-FEM discretization}\label{s:space-time}
Consider the neural field equations in the domain $\Omega.$ We will 
not distinguish between space and time variables and consider directly the space $\R ^{d+1},$ where 
$d$ is the number of space dimensions. 

Let $\mathcal{E}\subset\R^{d+1}$ be an open, bounded 
space-time domain in which a point has coordinates $(t,x)\in \R^{d+1},$ with 
${x}\in\R^d$ the position vector and time $t.$ 
First, partition the time interval 
$\bar I=[0,T]$ using the time levels $0=t_0<t_1<\ldots<t_{N}=T$ and denote by $I_n=(t_{n-1},t_n]$ the 
$n$-th time interval of length $k_n=t_n-t_{n-1}.$ 
A space-time slab is defined as $\mathcal{E}^n=I_n\times\Omega.$ 
Second, we approximate the spatial domain $\Omega$ with $\Omega_h$ using a tessellation of non-overlapping 
hexahedral elements (line elements in 1D, quadrilaterals in 2D, etc.)
\[
\bar{\mathcal{T}}_h=\left\{ K_j : \bigcup_{j=1}^{M} \bar{K}_j=\bar{\Omega}_h,\ 
K_j\cap K_i=\emptyset\ \text{if}\ i\not= j\right\}.
\]
The domain approximation is such that $\Omega_h\to\Omega$ 
as $h\to 0,$ where $h$ is the radius of the smallest sphere containing each element $K_j\in\bar{\mathcal{T}}_h.$ 
The space-time elements $\mathcal{K}_j^n$ are now obtained as $\mathcal{K}_j^n=(t_{n-1},t_n)\times K_j$. 
The space-time tessellation is defined as
\[
\mathcal{T}_h^n=\left\{ \mathcal{K}=G_\mathcal{K}^n(\hat{\mathcal{K}}) : K\in \bar{\mathcal{T}}_h\right\},
\]
where $G^n_\mathcal{K}$ denotes the mapping from the space-time
reference element $\hat{\mathcal{K}}=(-1,1)^{d+1}$ to the space-time element in physical space $\mathcal{K},$ 
see Fig.~\ref{fig:static_mesh}. The tessellation $\mathcal{T}_h$ of the whole discrete space-time domain is 
$\mathcal{T}_h=\cup_{n=1}^N\mathcal{T}_h^n.$
%%%%%%%%%%%%%%%%%%%%%%%
\begin{figure}%[htbp]
\begin{center}
\includegraphics[scale=.60]{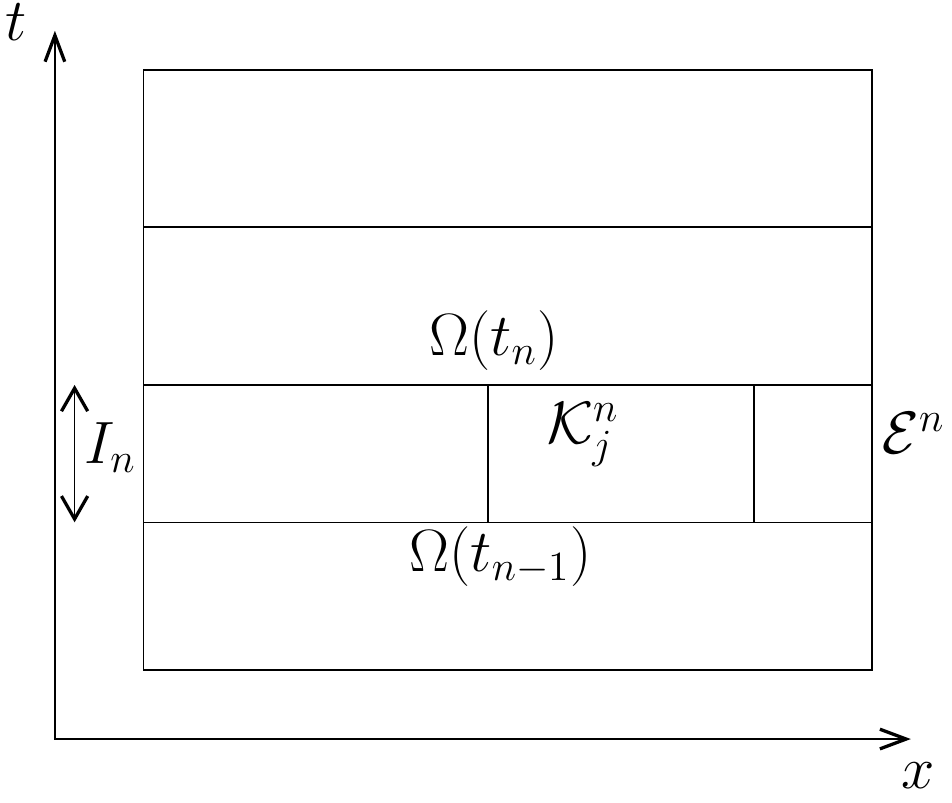}
\caption{Two-dimensional space-time elements in physical space.}
\label{fig:static_mesh}
\end{center}
\end{figure}
%%%%%%%%%%%%%%%%%%%%%

%%%%%%%%%%%%%%%%%%%%%%%%%%%%%%%%%%%%%%%%%%%%%
%\subsection{Time-discontinuous Galerkin finite element approximation}
The space-time FEM discretization is obtained by approximating the test and trial functions in 
each space-time element in the tessellation $\mathcal{K}^n\in\mathcal{T}_h^n$ with polynomial expansions 
that are assumed to be continuous within each space-time slab, but
discontinuous across the interfaces of the space-time slabs, namely at times $t_0,
t_1,\ldots , t_{N}.$ 

The finite element space associated with the tessellation $\mathcal{T}_h^n$ is 
defined as:
\begin{align}\label{FEM_space_slab}
V_h^n&=\bigl\{u\in  C(\mathcal{E}^n) :
u\mid_{\mathcal{K}}\circ\, G^n_{\mathcal{K}}\in \left(\hat{\mathcal{P}}_q(-1,1)\otimes
\hat{\mathcal{P}}_{r}(\hat{K})\right),
\forall\,\mathcal{K}\in\mathcal{T}_h^n\bigr\},
\end{align}
where $\hat{\mathcal{P}}_q(-1,1)$ and $\hat{\mathcal{P}}_{r}(\hat{K}),$ respectively, represent $q$th-order 
polynomials on $(-1,1)$ and $r$th-order tensor product polynomials in the reference element $\hat{K}=(-1,1)^d$. 
Finally, define 
\[
V_h=\{u\in L^2(\mathcal{E}) : u\mid_{\mathcal{E}_n}\in V_h^n,\ n=1,2,\dots,N\}.
\]
Note that the functions in $V_h$ are allowed to be discontinuous at the nodes of the partition of the time 
interval. We will use the notations $u^{n,\pm}=\lim_{s\to 0^\pm}u(t_n+s).$ 
Moreover, since $0\not\in I_1,$ we specify $u^{0,-}=u_0(0).$ 

The space-time dGcG-FEM method applied to problem (\ref{weak_formC})-\eqref{initial_condition} 
can be formulated as: find $u_h\in V_h$ such that 
\small
\begin{align}\label{global_weak}
&\sum_{n=1}^N\sum_{\mathcal{K}\in\mathcal{T}_h^n}\left[\Bigl(\ppf{u_h}{t}+\alpha u_h,v_h\Bigr)_{\mathcal{K}}-\int_{\mathcal{K}}\left[\int_\Omega J( x ,r)S\left(u_h(t-\tau( x ,r),r)\right)d\, r\right] v_h(t, x )dx \,dt\right]\nonumber\\%[5pt]
&+\sum_{n=2}^{N}\left( [u_h]_{n-1},v_h^{n-1,+}\right)
+\left(u_h^{0,+},v_h^{0,+}\right)=\left(u_0(0),v_h^{0,+} \right)
\end{align}
\normalsize
holds for all $v_h\in V_h$ and where $u_h^{0,-}=u_0(0).$ Here $[u_h]_n=u_h^{n,+}-u_h^{n,-}$ denotes the 
jump of $u_h$ at $t_n$ and $(\cdot,\cdot)_{\mathcal{K}}$ is the $L^2(\mathcal{K})$-inner product on a 
space-time element. The jumps were added to the weak 
formulation to ensure weak continuity between time slabs, since the basis functions 
in dGcG-FEM discretizations are discontinuous at the space-time slab boundary. 

Note that throughout this paper the FEM solution will be denoted by $u_h,$ which 
should not be confused with the state of the system notation introduced in Section \ref{s:intro}.  Moreover, it is 
important to remark that, for $u_h\in V_h,$ the segments ${u_h}_t,$ $t>0$ are not necessarily continuous, 
but piecewise continuous on $[-\tau_{max},0].$ Denoting the space of piecewise continuous functions on $[-\tau_{max},0]$ by 
$\hat X=PC\left([-\tau_{max},0];Y\right),$ we define the operator $\hat G: \hat X\to Y$ as
\begin{equation}\label{linear_operatorLhat}
\hat G\psi =\int_\Omega J(\cdot,r)S\left(\psi(-\tau(\cdot,r),r)\right)dr,\quad \psi\in\hat X.
\end{equation}
Then the nonlinear integral operator in \eqref{global_weak} is equal to $\hat G({u_h}_t).$

The weak formulation (\ref{global_weak}) can be transformed into an integrated-by-parts form, and since we 
added the jump term at each time level, it is possible to drop the summation over the 
space-time slabs. Moreover, after integration by parts, (\ref{global_weak}) can be decoupled into a sequence of local problems 
by choosing test functions that have support only in a single space-time slab $\mathcal{E}^n$. 
Hence we can solve the problem successively, i.e., 
using the known value $u_h(t_{n-1}^-)$ from the previous space-time slab. 
The weak formulation for the dGcG-FEM discretization of the neural field equation is the following:

Find $u_h\in V_h^n,$ such that for all $v_h\in V_h^n$ the variational equation is satisfied:
%\allowdisplaybreaks
\begin{align}\label{weak-form}
&\int_{\mathcal{K}^n}\left(-u_h\ppf{v_h}{t}+\alpha u_h v_h\right)dx dt
+\int_{K(t_{n})}u_h^{n,-}v_h^{n,-} dx\nonumber\\[7pt]
&-\int_{\mathcal{K}^n}\hat G\left({u_h}_t\right)(x) v_h dx dt=\int_{K(t_{n-1})}u_h^{n-1,-}v_h^{n-1,+} dx,
\end{align}
with $\mathcal{K}^n\in\mathcal{T}_h^n$ for $n=1,\dots,N$. 

Note here that the delay term may use values from space-time slabs where the solution was computed previously, 
but also from the current space-time slab, depending on the magnitude of the delay function compared to the time step. 
This problem will be discussed later in detail. 
%%%%%%%%%%%%%%%%%%%%%%%%%%%%
\subsection{How to treat the delay term?}\label{s:how to}
In this section we discuss the dGcG-FEM approximation of the delay term in the weak formulation 
(\ref{weak-form}). Introduce the approximation 
\begin{equation}\label{linear_eq}
u_h(t, x )\mid_{\mathcal{K}}=\sum_{m=1}^{N_p} \hat{u}_m^\mathcal{K}\psi_m^\mathcal{K}(t, x )
\end{equation}
into (\ref{weak-form}) and set the test function $v_h(t, x )\mid_{\mathcal{K}}=\psi_i^\mathcal{K}(t, x ),$ 
$i\in\{1,\dots,N_p\},$ with $N_p$ the number of degrees of freedom in element $\mathcal{K}$ and 
 $\psi_i^\mathcal{K}$ standard Lagrange tensor product basis functions. The delay term becomes
\begin{align}\label{delay term}
&\int_{\mathcal{K}}\psi_i^\mathcal{K}(t, x )\left(\int_\Omega J(x,r) S\left(u_h(t-\tau(x,r),r)\right) d r\right) d x\,d t \nonumber\\[7pt]
&={\int_{\mathcal K}} \psi_i^{\mathcal K}(t, x )
\Bigl(\sum_{L\in\bar{\mathcal{T}}_h} \int_L J(x,r)S\Bigl(\sum_{m=1}^{N_p}\hat{u}_m^L\psi_m^L(t-\tau(x,r),r)\Bigr) d r\Bigr) dx dt.
% &=\sum_{L\in\bar{\mathcal{T}}_h}\sum_{m=1}^{N_p} \hat{u}_m^L
% {\int_{\mathcal K}} \psi_i^{\mathcal K}(t, x )
% \int_L J(x,r)\psi_m^L(t-\tau(x,r),r) d r\, d x dt.
\end{align}
All integrals in the weak formulation are evaluated using Gaussian quadrature rules. Let us fix 
a quadrature point $(t_q,x_q)\in \mathcal{K}^n$ in a space-time element and let 
$\tau_{max}=\max_{(x,r)\in\bar\Omega\times\bar\Omega}\tau(x,r),$ as before.  
To compute the integral over a space element $L$ in \eqref{delay term}, consider a space quadrature point 
$r_{qs}\in\Omega,$ and 
distinguish three cases for the time delay $t_q-\tau(x_q,r_{qs})$, see Figure \ref{fig:grid}:

\medskip

{\bf Case 1.} If $-\tau_{max} \leq t_q-\tau(x_q,r_{qs})\leq 0,$ then the solution at this time level is given by 
the initial solution, i.e., $u_h(t_q-\tau(x_q,r_{qs}),r_{qs})=u_0(t_q-\tau(x_q,r_{qs}),r_{qs}).$

\medskip

{\bf Case 2.} When $t_q-\tau(x_q,r_{qs})\geq t_{n-1},$ then the delay term \eqref{delay term} is implicit since we 
remain in the same space-time slab $\mathcal{E}^n,$ where the solution is unknown. Hence, when the delay time is small 
enough compared to the time step, an additional Newton method needs to be incorporated for the solution of the nonlinear system. 

If we introduce the finite element approximations for $u_h$ and $v_h$ also into the other terms in the 
weak formulation (\ref{weak-form}), then we obtain for all $\mathcal{K}\in\mathcal{T}_h^n$
\begin{align}\label{weak-form-general}
&\sum_{j=1}^{N_p}\left\{\hat{u}_j^\mathcal{K}\int_{\mathcal{K}}\Bigl(-\psi_j^{\mathcal K}(t, x )\frac{\partial }{\partial t}\psi_i^{\mathcal K}(t, x )+\alpha \psi_j^{\mathcal K}(t, x )\psi_i^{\mathcal K}(t, x )\Bigr)dx dt\right.\nonumber\\[7pt]
&\qquad+\left.\hat{u}_j^\mathcal{K}\int_{K(t_{n}^-)}\psi_j^{\mathcal K}(t_{n}^-, x )\psi_i^{\mathcal K}(t_{n}^- , x )dx\right\}\nonumber\\[7pt]
&-\sum_{L\in\bar{\mathcal{T}}_h}{\int_{\mathcal K}} \psi_i^{\mathcal K}(t, x )
\Bigl[\int_L J(x,r)S\Bigl(\sum_{m=1}^{N_p} \hat{u}_m^L\psi_m^L(t-\tau(x,r),r)\Bigr) d r\Bigr] dx dt \nonumber\\[7pt]
&=\sum_{j=1}^{N_p}\hat{u}_j^{\mathcal{K},n-1}\int_{K(t_{n-1}^+)}\psi_j^{\mathcal K}(t_{n-1}^-, x )\psi_i^{\mathcal K}(t_{n-1}^+, x )dx,
\end{align}
where $ \hat{u}_j^{\mathcal{K},n-1}$ are the coefficients of space-time element $\mathcal{K}$ in the space-time slab $\mathcal{E}^{n-1}.$

\medskip

{\bf Case 3.} When $0\leq t_q-\tau(x_q,r_{qs})< t_{n-1},$ then the delay term is explicit since we go 
back to a previous space-time slab, where the FEM solution is already computed. 
%%%%%%%%%%%%%%%%%%%%%%%%%%
\begin{figure}%[htbp]
\begin{center}
\includegraphics[scale=.4]{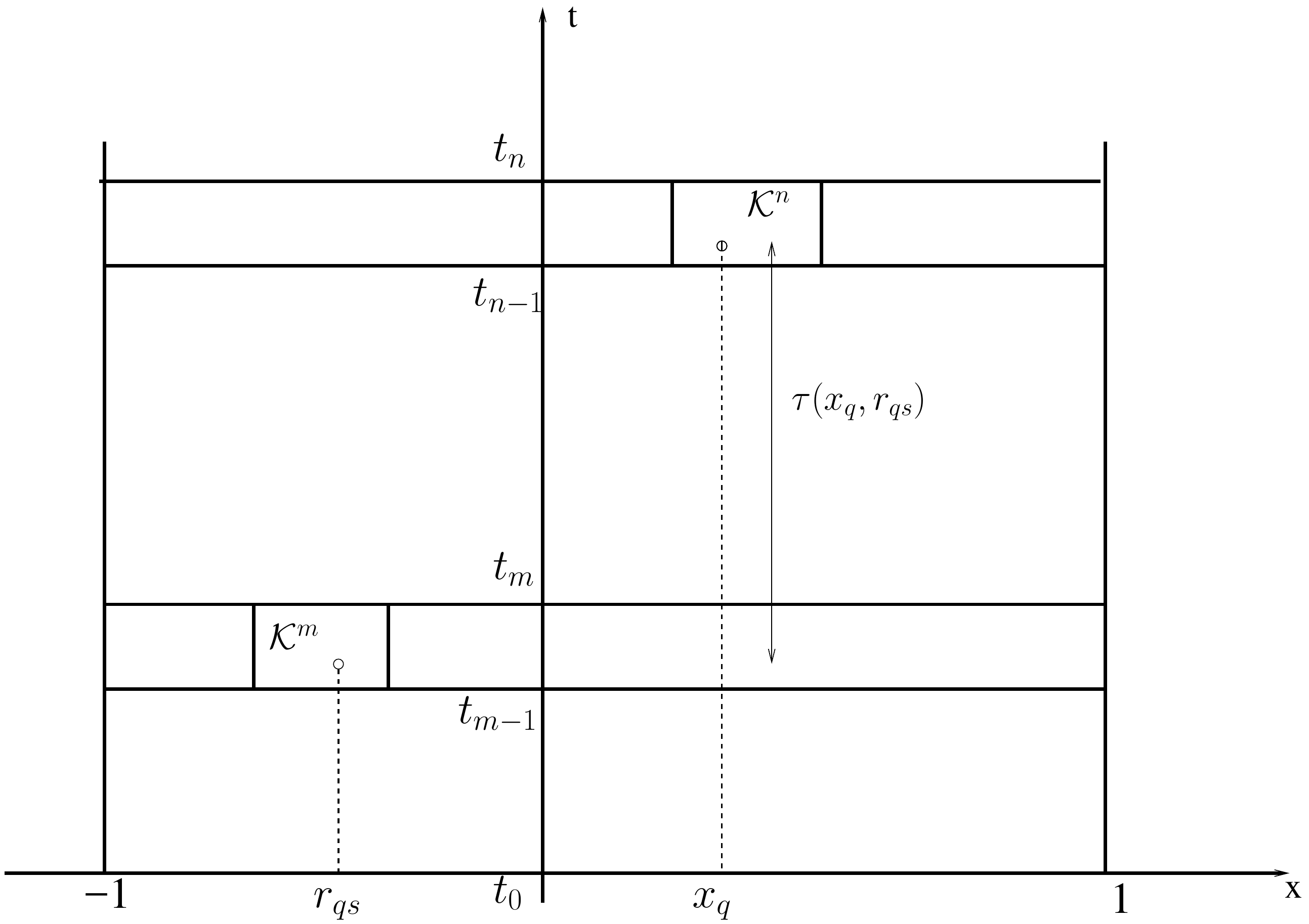}
\caption{The computational domain.}
\label{fig:grid}
\end{center}
\end{figure}
\endgroup
%%%%%%%%%%%%%%%%%%%%%%%%%%%%%%%%%%%%%%%%%%%%%%%%%%%
\section{Error analysis}\label{s:error analysis}
In this section we give an a-priori error analysis for the space-time dGcG-method (\ref{weak-form-general}). 
In the error analysis we will use a slightly modified version of the temporal interpolation functions defined in Proposition 4.1, \cite{thomee_book}. 
First, define the space
\begin{equation}
S_k=\{ w:[0,T]\to Y : w\mid_{I_n}=\sum_{j=0}^ q \varphi_j t^j,\ \varphi_j\in Y,\ \forall n\geq 1\},
\end{equation}
with $I_n=(t_{n-1},t_n]$ and $|I_n|=k_n.$ Note that these functions are allowed to be discontinuous at the nodes of the partition of the time interval, 
but continuous from the left in each subinterval $I_n,$ i.e., $w(t_n)=\lim_{t\to t_n^-}w(t).$ For the 
restriction of the functions in $S_k$ to $I_n,$ we use the notation $S_k^n.$ Define the temporal polynomial 
interpolant
\begin{equation}
T_k: C\left([0,T],Y\right)\to S_k
\end{equation}
as follows, see also \cite{thomee_book}.
\begin{prop}\label{prop:interpolant} 
Let $\tilde{u}=T_k u\in S_k$ be the time-interpolant of $u\in C\left([0,T],Y\right)\cap H^{q+1}\left([0,T],Y\right),$ $q\geq 0,$ 
with the following properties:
\begin{alignat}{2}
&\tilde u(t_{n-1})=u(t_{n-1}),&& \text{for  } n\geq 1\\[5pt]
\int_{I_t}&\left(\tilde u(s)-u(s)\right)s^l ds=0, &&\text{ for  } l=0,\dots, q-1,\ t\in I_n,\ 
I_t=(t_{n-1},t],\ n\geq 1. \label{interpolant}
\end{alignat} 
The interpolation error then can be estimated as
\begin{equation}\label{interpolation_estimate}
\|\tilde u(s)-u(s)\|\leq C_I k_t^{q+1/2}\left(\int_{I_t}\|\partial_{s}^{q+1}u(s,\cdot)\|^2 ds\right)^{1/2}, \text{ for } s\in I_t,
\end{equation}
where $\partial_s^{q+1}$ denotes the $(q+1)$-th order derivative w.r.t. time, $k_t=|I_t|$ 
and the norm $\Vert\cdot\Vert=\Vert\cdot\Vert_{L^2(\Omega)}$ hereafter.
\end{prop}
Observe that $\tilde u$ interpolates exactly at the nodes and the interpolation error is orthogonal to 
polynomials of degree at most $q-1.$ 
For constant polynomials ($q=0$) condition \eqref{interpolant} is not used.

Next, define the spatial interpolant. Let $W_h$ be the space of tensor
product polynomials of degree up to $r\geq 0$ on each space element
$K_j,$ i.e.,
\begin{equation}
W_h=\{ v\in C(\Omega): v\mid_{K}\circ \,G^n_{K}\in \hat{\mathcal{P}}_r(\hat K), \forall K\in\bar{\mathcal{T}}_h\},
\end{equation}
where $G^n_K$ denotes the mapping from the reference element $\hat{K}=(-1,1)^{d}$ to the element $K\in\bar{\mathcal{T}}_h$ in physical space. Let 
\[
P_h: Y\to W_h
\]
be the $L^2$-projection to the (spatial) finite element space, defined as $\left(P_h v,w_h\right)=\left(v,w_h\right)$ for all $w_h\in W_h.$ We use the standard interpolation estimate in space (see e.g.~\cite{Brenner}, \cite{ciarlet})
\begin{equation}
\Vert v-P_h v\Vert\leq C  h^{r+1}\|v\|_{r+1}\quad \forall v\in Y\cap H^{r+1}(\Omega), 
\end{equation}
where $\|\cdot\|_{r+1}=\|\cdot\|_{H^{r+1}(\Omega)},$ $ h$ denotes the maximal space element 
diameter as before, and the constant $C$ is independent of $ h$ and $v.$ 

In the error analysis we also need the interpolation of the initial segment of the solution. 
Let the given initial function be $u_0\in X\cap H^{q+1}\left([-\tau_{max},0];H^{r+1}(\Omega)\right)$ 
for some $q,r\geq 0.$ Use a partition of the interval $[-\tau_{max},0]$ into $M$ subintervals $J_i$ of length $k_i,$ respectively. 
On each $J_i$ we use the same temporal interpolation $\tilde u_0=T_k u_0$ of $u_0,$ as introduced in 
Proposition \ref{prop:interpolant}. Then for all $s\in J_i$ we have 
\begin{align}\label{interpolation_estimate_psi2}
\Vert u_0(s)&-P_h T_k u_0(s)\Vert =\Vert u_0(s)-P_h \tilde{u}_0(s)\Vert\nonumber\\[5pt]
&\leq \Vert u_0(s)-P_h u_0(s)\Vert+\Vert P_h\Vert \Vert u_0(s)-T_k u_0(s)\Vert\nonumber\\[5pt]
&\leq C h^{r+1}\|u_0(s)\|_{r+1}+ C_I k_i^{q+1/2}\left(\int_{J_i}\|\pd{s}{q+1}u_0(s,\cdot)\|^2 ds\right)^{1/2},
\end{align}
where we use that the operator norm of the Lagrange interpolation $P_h$ is bounded, see \cite{Brenner}, 
\cite{ciarlet_article}.

We will also need an estimate of the integral of the interpolation error on the partition of the initial segment. There exists $C>0$ generic constant (independent of the solution and mesh size), such that
\begin{align}\label{interpolation_estimate_psi_int}
\int_{J_i}&\Vert u_0(s)-P_h T_k u_0(s)\Vert^2 ds\leq C\mathcal{B}(u_0,J_i)\nonumber\\
&:=C  \left(h^{2r+2}k_i \|u_0\|^2_{r+1,J_i}+k_i^{2q+2}\int_{J_i}\Vert \pd{s}{q+1}u_0(s,\cdot)\Vert^2 ds\right),
\end{align}
where we denoted the norm
\[
%\Vert\varphi\Vert_{I_n}=\sup_{t\in I_n}\Vert \varphi(t)\Vert, \quad\text{and } 
\|\varphi\|_{r,I}=\sup_{t\in I}\|\varphi(t)\|_r.
\]
%%%%%%%%%%%%%%%%%%%%%%%%%%%%%%%%%

Next, we state the main result of the a-priori error analysis of the dGcG discretization \eqref{weak-form} for the neural field equations.
\begin{thm}
Let $u\in C^1\left([0,T);Y\right)\cap H^{q+1}\left([0,T];H^{r+1}(\Omega)\right)$ be the solution of (\ref{DDE_G}) for some $q,r\geq 0$ and with initial state $u_0\in X\cap H^{q+1}\left([-\tau_{max},0];H^{r+1}(\Omega)\right),$ 
and let $u_h\in V_h^n$ be the solution of (\ref{weak-form}). Then 
\begin{align}
\| u_h(t_N)-u(t_N)\|^2\leq C &\left(\sum_{i=1}^M m(i)\mathcal{B}(u_0,J_i)+\sum_{n=1}^N m(n)k_n^{2q+2}\int_{I_n}\|\pd{t}{q+1}u(t,\cdot)\|^2 dt\right.\nonumber\\
&\left.+ h^{2r+2}\sum_{n=0}^N \|u(t_n)\|^2_{r+1}+\sum_{n=1}^N h^{2r+2}m(n)k_n \|u\|^2_{r+1,I_n}\right)
\end{align}
holds for $t_N\geq 0,$ $N$ the number 
of time slabs, where $C$ is a positive constant independent of the time step $k_n=t_n-t_{n-1}$ 
and the maximal space element diameter $h$. Here $m(n)\leq N-1,$ $1\leq n\leq N,$ is the multiplicity how many times we visited the interval $I_n$ due to the delay term.
\end{thm}
\begin{proof}
Let us decompose the error of the numerical discretization into the sum 
\begin{align}
(u_h-u)(t,x)&=\left[u_h(t,x)-P_h\tilde u(t)(x)\right]+\left[P_h\tilde u(t)(x)-u(t,x)\right]\nonumber\\[5pt]
&=\theta(t,x)+\rho(t,x) \qquad \text{for } t>0,  
\end{align} 
with $\theta$ the discretization error and $\rho$ the interpolation error. 
When $t\in[-\tau_{max},0],$ we only have the interpolation error of the given initial solution 
$u_0,$ that is, $\theta(t,x)=0$ and $\rho(t,x)=P_h\tilde u_0(t)(x)-u_0(t,x).$ 
From here on, we suppress the spatial dependence where it is clear from the context. 
Since $\tilde u$ interpolates exactly at the nodes $t=t_{n-1},$ we have that 
\begin{align}\label{interpolant_space}
\Vert\rho(t_{n-1})\Vert&=\Vert P_h T_k u(t_{n-1})-u(t_{n-1})\Vert\nonumber\\
&=\Vert P_h u(t_{n-1})-u(t_{n-1})\Vert\leq C  h^{r+1}\| u(t_{n-1})\|_{r+1}
\end{align} 
holds for all $n\geq 1.$ Here the constant $C$ is independent of $ h,$ see e.g. \cite{Brenner}. 
When we are in the interior of a time 
interval $I_j$, we decompose $\rho$ to be able to use the bound on the interpolation error in time and space, respectively, as in \eqref{interpolation_estimate_psi2}
\begin{align}\label{rho-interpolant}
\Vert\rho(t)\Vert &=\| P_h T_k u(t)-u(t)\| \nonumber\\[5pt]
&\leq C\left(  h^{r+1} \|u(t)\|_{r+1} +k_j^{q+1/2}\Bigl(\int_{I_j}\Vert \pd{s}{q+1}u(s,\cdot)\Vert^2 ds
\Bigr)^{1/2}\right), 
\end{align}
for any $t\in I_j$ and $j=1,\dots,N.$
It is, therefore, sufficient to bound $\theta^N=\theta(t_N).$ Since both $u_h$ and $u$
satisfy the weak formulation (\ref{weak-form}) with $G$ and $\hat G,$ respectively, 
we obtain that for all $v\in V_h^n$ 
\begin{align}\label{weak-form-error}
\int_{I_n}&\left(\ppf{}{t}\theta(t)+\alpha\theta(t),v(t)\right) dt+\left([\theta]_{n-1},v^{n-1,+}\right)\nonumber\\
&=\int_{I_n}\left(-\ppf{}{t}\rho(t)-\alpha\rho(t)+\hat{G}({u_h}_t)-G(u_t),v(t)\right) dt-\left([\rho]_{n-1},v^{n-1,+}\right).
\end{align}
The variational equation (\ref{weak-form-error}) holds for any partition of the time interval $I_n$, 
hence the following equation is also valid for any $t\in (t_{n-1},t_n]$
\begin{align}\label{weak-form-error-t}
\int_{I_t}&\left(\ppf{}{s}\theta(s)+\alpha\theta(s),v(s)\right) ds+\left([\theta]_{n-1},v^{n-1,+}\right)\nonumber\\
&=\int_{I_t}\left(-\ppf{}{s}\rho(s)-\alpha\rho(s)+\hat{G}({u_h}_s)-G(u_s),v(s)\right) ds-\left([\rho]_{n-1},v^{n-1,+}\right),
\end{align}
where $I_t=(t_{n-1},t].$ Using the assumptions on the interpolant, some terms in (\ref{weak-form-error-t}) will cancel, i.e., 
for all $t\in I_n$
\begin{align}
\int_{I_t}&\left(\ppf{}{s}\rho(s),v(s)\right)ds+\left(\rho(t_{n-1}^+)-\rho(t_{n-1}^-),v^{n-1,+}\right)\nonumber\\
&=-\int_{I_t}\left(\rho(s), \ppf{}{s}v(s)\right)ds+ \left(\rho(s), v(s)\right)|_{t=t_{n-1}^+}^{t}
+\left(\rho^{n-1,+}-\rho^{n-1,-},v^{n-1,+}\right)\nonumber\\
&= -\int_{I_t}\left(\rho(s),\ppf{}{s}v(s)\right)ds+\left(\rho(t),v(t)\right)-\left(\rho^{n-1,-},v^{n-1,+}\right)\nonumber\\[5pt]
&=\left(\rho(t),v(t)\right)-\left(\rho^{n-1,-},v^{n-1,+}\right).
\end{align}
Let $v=2\theta\in S_k^n$ in (\ref{weak-form-error-t}). Then for each $I_n$ and $t\in I_n$ the following holds 
\begin{align}%\label{weak-form-error-test}
\int_{I_t}&2\left(\ppf{}{s}\theta(s)+\alpha\theta(s),\theta(s)\right) ds+2\left([\theta]_{n-1},\theta^{n-1,+}\right)
\nonumber\\
&=\int_{I_t}2\left(-\alpha\rho(s)+\hat{G}({u_h}_s)-G(u_s),\theta(s)\right) ds -2\left(\rho(t),\theta(t)\right)+2\left(\rho^{n-1,-},\theta^{n-1,+}\right) .
\end{align}
This may be further written as 
\begin{align}\label{weak-form-error-test}
\int_{I_t}&\Bigl[\frac{d}{ds}\Vert\theta(s)\Vert^2 +2\alpha\Vert\theta(s)\Vert^2\Bigr] ds
+2\Vert\theta^{n-1,+}\Vert^2=2\left(\theta^{n-1,-},\theta^{n-1,+}\right)
\nonumber\\
+&\int_{I_t}2\left(-\alpha\rho(s)+\hat{G}({u_h}_s)-G(u_s),\theta(s)\right) ds -2\left(\rho(t),\theta(t)\right)+2\left(\rho^{n-1,-},\theta^{n-1,+}\right) .
\end{align}
% We will use the following estimates to bound the error:
% \begin{itemize}
% \item[(a)] For any $\epsilon>0$ small, 
% \begin{equation}\label{estimate_epsilon}
% 2|\left(\rho(t),\theta(t)\right)|\leq \frac{1}{\epsilon^2}\|\rho(t)\|^2+\epsilon^2 \|\theta(t)\|^2,
% \end{equation}
% \item[(b)] when $\epsilon=\sqrt{2}$ as above, 
% \begin{equation}
% 2|\left(\theta^{n-1,-},\theta^{n-1,+}\right)|\leq 2\|\theta^{n-1,-}\|^2+\frac{1}{2}\|\theta^{n-1,+}\|^2,
% \end{equation}
% \item[(c)] when $\epsilon=1$
% \begin{equation}
% 2|\left(\rho(s),\theta(s)\right)|\leq \|\rho(s)\|^2+ \|\theta(s)\|^2\text{  for  } s\in I_t, 
% \end{equation}
% \item[(d)] when $\epsilon=\sqrt{2}$
% \begin{equation}
% 2|\left(\rho^{n-1,-},\theta^{n-1,+}\right)|\leq 2\|\rho^{n-1,-}\|^2+ \frac{1}{2}\|\theta^{n-1,+}\|^2. 
% \end{equation}
% \end{itemize}
% Using (a)-(d) in (\ref{weak-form-error-test}), we obtain
Using the Schwarz inequality and the inequality $2ab\leq \epsilon^2a^2+\frac{1}{\epsilon^2}b^2$ we 
obtain
\begin{align}\label{weak-form-error-1}
(1-\epsilon^2)\Vert\theta(t)\Vert^2\leq  & -2\alpha\int_{I_t}\Vert\theta(s)\Vert^2 ds 
+2 \Vert\theta^{n-1,-}\Vert^2\nonumber\\
& +\alpha  \int_{I_t} \left(\Vert\rho(s)\Vert^2+\Vert\theta(s)\Vert^2\right)ds 
+\frac{1}{\epsilon^2}\|\rho(t)\|^2 \nonumber\\[5pt]
&+ 2\int_{I_t} \left(\hat{G}({u_h}_s)-G(u_s),\theta(s)\right) ds+ 2\|\rho^{n-1,-}\|^2.
\end{align}
Since the nonlinearity $S$ is Lipschitz continuous with some Lipschitz constant $C_S,$ we can estimate the 
nonlinear term as
\begin{align}\label{nonlinear_estimate}
&2\int_{I_t}\left( \hat{G}({u_h}_s)-G(u_s), \theta(s)\right)ds\nonumber\\
&=2\int_{I_t}\int_\Omega\left[\int_\Omega J(x,r)\left[S\left(u_h\left(s-\tau(x,r),r\right)\right)-S\left(u\left(s-\tau(x,r),r\right)\right)\right]dr \right]\theta(s,x)dx\, ds\nonumber\\
&\leq 2C_S \int_{I_t}\int_\Omega \left[\int_\Omega |J(x,r)|
\left(|\theta\left(s-\tau(x,r),r\right)|+|\rho\left(s-\tau(x,r),r\right)|\right)dr\right]\theta(s,x)dx\, ds.
\end{align}
Let us estimate the first term on the right hand side of (\ref{nonlinear_estimate}) as
\begin{align}\label{critical_term1}
T_1:=&\int_{I_t}\int_\Omega\left[\int_\Omega |J(x,r)||\theta\left(s-\tau(x,r),r\right)|dr \right]\theta(s,x)dx\, ds\nonumber\\[5pt]
&\leq \int_{I_t}\left(\int_\Omega\left(\int_\Omega |J(x,r)||\theta\left(s-\tau(x,r),r\right)|dr\right)^2 dx\right)^{1/2}
\left(\int_\Omega \theta^2(s,x)dx\right)^{1/2} ds\nonumber\\[5pt]
&\leq \int_{I_t}\left(|\Omega|\int_\Omega\int_\Omega J^2(x,r)\theta^2\left(s-\tau(x,r),r\right)dr dx\right)^{1/2}
\left(\int_\Omega \theta^2(s,x)dx\right)^{1/2} ds\nonumber\\[5pt]
&\leq \left(|\Omega|\int_{I_t}\int_\Omega\int_\Omega J^2(x,r)\theta^2\left(s-\tau(x,r),r\right)dr dx ds\right)^{1/2}
\left(\int_{I_t}\int_\Omega \theta^2(s,x)dx ds\right)^{1/2}
\end{align}
where we used the Schwarz inequality in each estimation step and $|\Omega|=\text{vol}(\Omega)$. 
Next, since $0<\tau(x,r)\leq \tau_{max},$ and $J(x,r)\leq \Vert J\Vert_C,$ for all $(x,r)\in\bar\Omega\times\bar\Omega,$ 
the following estimate is valid  
\begin{align}\label{key_idea}
\int_{I_t}\int_\Omega\int_\Omega &J^2(x,r)\theta^2(s-\tau(x,r),r)dr \,dx\, ds \nonumber\\[5pt]
&\leq \Vert J\Vert_C^2 \int_{I_t}\int_\Omega\int_\Omega\theta^2(s-\tau(x,r),r)dr\, dx \,ds \nonumber\\[5pt]
&\leq \Vert J\Vert_C^2 |\Omega|\int_{t_{n-1}-\tau_{max}}^{t}\int_\Omega\theta^2(s,r)dr\,ds
\end{align}
Hence we can further estimate (\ref{critical_term1}) as
\begin{align}\label{critical_term2}
&T_1\leq  \|J\|_C |\Omega| \left(\int_{t_{n-1}-\tau_{max}}^{t} \|\theta(s)\|^2 ds\right)^{1/2} \left(\int_{I_t} \|\theta(s)\|^2 ds\right)^{1/2}\nonumber\\[5pt]
&\leq \|J\|_C |\Omega| \int_{t_{n-1}-\tau_{max}}^{t} \|\theta(s)\|^2 ds= \|J\|_C |\Omega|\Bigl( \int_{t_{n-1}-\tau_{max}}^{t_{n-1}} \|\theta(s)\|^2 ds + \int_{I_t}\|\theta(s)\|^2 ds\Bigr) .
\end{align}
Similarly as in (\ref{critical_term1}) and (\ref{key_idea}), for the last term in (\ref{nonlinear_estimate}) we obtain
\begin{align}
T_2:=&2\int_{I_t}\int_\Omega\left[\int_\Omega |J(x,r)||\rho\left(s-\tau(x,r),r\right)|dr \right]\theta(s,x)dx\, ds\nonumber\\[5pt]
&\leq 2 \|J\|_C |\Omega|
\left(\int_{t_{n-1}-\tau_{max}}^{t}\Vert \rho(s)\Vert^2 ds \right)^{1/2}
\left(\int_{I_t}\Vert \theta(s)\Vert^2 ds\right)^{1/2}\nonumber\\[5pt]
&\leq  \|J\|_C^2 |\Omega|^2 \int_{t_{n-1}-\tau_{max}}^{t_n}\| \rho(s)\|^2 ds
+\int_{I_t}\|\theta(s)\|^2 ds.
\end{align}
After introducing the above estimates into (\ref{weak-form-error-1}) we obtain that for all $t\in I_n,$ 
\begin{align}\label{applyG_lemma}
(1-\epsilon^2)\Vert\theta(t)\Vert^2\leq  & (C_S-\alpha+2C_S\Vert J\Vert_C |\Omega|)\int_{I_t}\Vert\theta(s)\Vert^2 ds 
+ 2\Vert\theta^{n-1}\Vert^2\nonumber\\
& +\alpha  \int_{I_n} \Vert\rho(s)\Vert^2 ds +2 C_S\|J\|_C |\Omega|\int_{t_{n-1}-\tau_{max}}^{t_{n-1}} \|\theta(s)\|^2 ds \nonumber\\
&+ C_S \|J\|_C^2 |\Omega|^2 \int_{t_{n-1}-\tau_{max}}^{t_n}\| \rho(s)\|^2 ds +\frac{1}{\epsilon^2}\|\rho(t)\|^2+ 2\|\rho^{n-1}\|^2
\end{align}
is valid for all $n\geq 1.$ Divide by $1-\epsilon^2,$ where $0<\epsilon<1,$ and denote by 
\begin{align*} 
&\beta=\frac{|\,C_S-\alpha+2C_S\Vert J\Vert_C|\Omega|\ |}{1-\epsilon^2}>0\nonumber \\[5pt]
&\omega_n(t)=\gamma_n+\frac{1}{\epsilon^2(1-\epsilon^2)}\|\rho(t)\|^2\\[5pt]
&\gamma_n=\frac{1}{1-\epsilon^2}\left(2\Vert\theta^{n-1}\Vert^2
+\alpha  \int_{I_n} \Vert\rho(s)\Vert^2 ds +2 C_S\|J\|_C |\Omega|\int_{t_{n-1}-\tau_{max}}^{t_{n-1}} \|\theta(s)\|^2 ds\right. \nonumber\\[5pt]
&\quad \left.+ C_S \|J\|_C^2 |\Omega|^2 \int_{t_{n-1}-\tau_{max}}^{t_n}\| \rho(s)\|^2 ds +2\|\rho^{n-1}\|^2\right),\quad n\geq 1.
\end{align*}
Then, inequality (\ref{applyG_lemma}) can be written as
\begin{equation}\label{applyG_lemma1}
\eta(t)\leq \omega_n(t)+\beta\int_{I_t}\eta(s)ds, \ t\in I_n,
\end{equation}
where $\eta(t)=\Vert\theta(t)\Vert^2.$ Apply Gr\" onwall's inequality to (\ref{applyG_lemma1}) to obtain
\begin{equation}\label{G_lemma0}
\eta(t)\leq \omega_n(t)+\beta\int_{I_t}\omega_n(s)e^{\beta(t-s)}ds, \ t\in I_n.
\end{equation}
When $t=t_n,$ 
\begin{equation}\label{G_lemma}
\eta(t_n)\leq \omega_n(t_n)+\beta\int_{I_n}\omega_n(s)e^{\beta(t_n-s)}ds, 
\end{equation}
where 
\begin{align}
\omega_n(t_n)&=\gamma_n+\frac{1}{\epsilon^2(1-\epsilon^2)}\Vert\rho(t_n)\Vert^2,\quad n\geq 1.
%=\frac{1}{1-\epsilon^2}\left(2\eta(t_{n-1})
%+\alpha  \int_{I_n} \Vert\rho(s)\Vert^2 ds +2 \|J\|_C |\Omega|\int_{t_{n-1}-\tau_{max}}^{t_{n-1}} \eta(s) ds\right. \nonumber\\[5pt]
%&\ \left.+  \|J\|_C^2 |\Omega|^2 \int_{t_{n-1}-\tau_{max}}^{t_n}\| \rho(s)\|^2 ds+\frac{1}{\epsilon^2}\|\rho^n\|^2+2\|\rho^{n-1}\|^2 \right),\quad n\geq 1.
\end{align}
Note that the only time-dependent term in $\omega_n(t)$ is $\|\rho(t)\|^2.$ 
Hence, the integral term in (\ref{G_lemma}) can be estimated as
\begin{align}\label{integral_bound}
\int_{I_n}\omega_n(s)e^{\beta(t_n-s)}ds&=\int_{I_n}\left(\gamma_n+\frac{1}{\epsilon^2(1-\epsilon^2)}\Vert\rho(s)\Vert^2\right)e^{\beta(t_n-s)}ds\nonumber\\[5pt]
&\leq e^{\beta k_n}\left( k_n\gamma_n+\frac{1}{\epsilon^2(1-\epsilon^2)}\int_{I_n}\Vert \rho(s)\Vert^2 ds\right).
\end{align}
Therefore, we obtain for $n\geq 1$ that
\begin{align}\label{G_lemma1}
\eta(t_n)&\leq \left(1+\beta k_n e^{\beta k_n}\right)\gamma_n+\frac{\beta e^{\beta k_n} }{\epsilon^2(1-\epsilon^2)}\int_{I_n}\Vert \rho(s)\Vert^2 ds+\frac{1}{\epsilon^2(1-\epsilon^2)}\Vert \rho^n\Vert^2.%\nonumber\\
%&=\frac{1+\beta k_n e^{\beta t_n}}{1-\epsilon^2}\left(2\eta(t_{n-1})+2 \|J\|_C |\Omega|\int_{t_{n-1}-\tau_{max}}^{t_{n-1}} \eta(s) ds\right)\nonumber\\
%&+\frac{1+\beta k_n e^{\beta t_n}}{1-\epsilon^2}\left(\alpha  \int_{I_n} \Vert\rho(s)\Vert^2 ds +\|J\|_C^2 |\Omega|^2 \int_{t_{n-1}-\tau_{max}}^{t_n}\| \rho(s)\|^2 ds+2\|\rho^{n-1}\|^2 \right)\nonumber\\
%&+\frac{\beta e^{\beta t_n} }{\epsilon^2(1-\epsilon^2)}\int_{I_n}\Vert \rho(s)\Vert^2 ds.
\end{align}
Let us recall that 
\begin{align}
\gamma_n&=\frac{1}{1-\epsilon^2}\left(2\eta(t_{n-1})
 +2 C_S\|J\|_C |\Omega|\int_{t_{n-1}-\tau_{max}}^{t_{n-1}} \eta(s) ds+\alpha  \int_{I_n} \Vert\rho(s)\Vert^2 ds\right. \nonumber\\[5pt]
&\qquad\qquad  \left.+  C_S\|J\|_C^2 |\Omega|^2 \int_{t_{n-1}-\tau_{max}}^{t_n}\| \rho(s)\|^2 ds+2\|\rho^{n-1}\|^2 \right)
\end{align}
and observe that the right hand side of \eqref{G_lemma1} can be estimated by the bound of the 
interpolation error and the bound of the integral of $\eta(t)$ over earlier time intervals, i.e., for $t\leq t_{n-1}.$ Hence we can write
\begin{align}\label{recursion_eta}
\eta(t_n)&\leq C_1\eta(t_{n-1})+C_2\int_{t_{n-1}-\tau_{max}}^{t_{n-1}} \eta(s) ds\nonumber\\
&+ C_3\int_{t_{n-1}-\tau_{max}}^{t_n}\| \rho(s)\|^2 ds+C_4\|\rho^{n-1}\|^2 +\frac{1}{\epsilon^2(1-\epsilon^2)}\|\rho^{n}\|^2, 
\end{align}
where $C_i,$ $i=1,\dots,4$ depend on the parameters $\alpha, \beta, \epsilon, \Vert J\Vert_C,$ 
$|\Omega|$ and $k_n,$ such that $C_i=O(1)$ as $k_n\to 0.$
 
By integrating (\ref{G_lemma0}) we obtain the following general formula
\begin{align}\label{general_formula}
\int_{t_{n-1}-\tau_{max}}^{t_{n-1}}\eta(s)ds&\leq \sum_{j=m(n)}^{n-1}\int_{I_j}\eta(s)ds\nonumber\\
&\leq\sum_{j=m(n)}^{n-1}\int_{I_j}\left(\omega_j(s)+\beta\int_{I_s}\omega_j(\tau)e^{\beta(s-\tau)}d\tau\right)ds\nonumber\\
&\leq\sum_{j=m(n)}^{n-1}\int_{I_j}\left(\omega_j(s)+\beta\int_{I_j}\omega_j(\tau)e^{\beta(t_j-\tau)}d\tau\right)ds\nonumber\\
%&\leq\sum_{j=m(n)}^{n-1}\int_{I_j}\omega_j(s)ds+k_j\beta\int_{I_j}\omega_j(\tau)e^{\beta(t_j-\tau)}d\tau\nonumber\\
&\leq\sum_{j=m(n)}^{n-1}\left(1+k_j \beta e^{\beta k_j}\right)\int_{I_j} \omega_j(s)ds\nonumber\\
&\leq\sum_{j=m(n)}^{n-1}\left(1+k_j \beta e^{\beta k_j}\right)\left(k_j\gamma_j+\frac{1}{\epsilon^2(1-\epsilon^2)}\int_{I_j} \Vert\rho(s)\Vert^2ds\right),
\end{align}
where we used that $\eta(s)=\Vert\theta(s)\Vert^2=0$ for $s\in[-\tau_{max},0]$ and 
\eqref{integral_bound} in the last inequality. Here $m=m(n)\leq n-1$ 
is the index of the interval $I_m$ for which $t_{n-1}-\tau_{max}\in I_m.$  

As we can see, the integral of $\eta(s)$ can be bounded by the integral of $\Vert\rho(s)\Vert,$ hence in 
\eqref{recursion_eta} we have
\begin{align}
  \eta(t_n)&\leq C\sum_{j=m(n)}^{n-1}\left[\left(1+k_j \beta e^{\beta k_j}\right)k_j\gamma_j+ 
  \left(\frac{1+k_j \beta e^{\beta k_j}}{\epsilon^2(1-\epsilon^2)}+1\right)\int_{I_j}\Vert\rho(s)\Vert^2 ds\right]\nonumber\\[5pt]
&+C_3\int_{I_n}\Vert\rho(s)\Vert^2 ds+C_4\|\rho^{n-1}\|^2 +\frac{1}{\epsilon^2(1-\epsilon^2)}\|\rho^{n}\|^2+C_1\eta(t_{n-1}). 
\end{align}
We can use \eqref{rho-interpolant} to bound the integral of 
$\Vert\rho(s)\Vert$ as follows
\begin{align}\label{rho_past}
\int_{I_j}\| \rho(s)\|^2 ds= C\left[ h^{2r+2}k_j \|u\|^2_{r+1,I_j} +k_j^{2q+2}\int_{I_j}
\Vert \pd{s}{q+1}u(s,\cdot)\Vert^2 ds\right]
% &\int_{t_{n-1}-\tau_{max}}^{t_n}\| \rho(s)\|^2 ds=\int_{t_{n-1}-\tau_{max}}^{t_n}\| P_h\tilde u(s)-u(s)\|^2 ds\nonumber\\
% &\leq \sum_{i=1}^M \int_{J_i}\Vert P_h\tilde u_0(s)-u_0(s)\Vert^2 ds+ 
% \sum_{j=1}^n \int_{I_j} \| P_h\tilde u(s)-u(s)\|^2 ds\nonumber\\
% &\leq 
% %\left[C  h^{2r}k \Vert u_0\Vert^2_{r,J_i} +C_Ik^{2q+2}\left(C  h^{2r}+1\right)\int_{J_i}\Vert u_0^{(q+1)}(s)\Vert^2 ds\right]\nonumber\\
% \sum_{j=1}^n C\left[  h^{2r}k_j |u|^2_{r,I_j} +k_j^{2q+2}\int_{I_j}\left(\Vert u^{(q+1)}(s)\Vert^2 +
%  h^{2r}|u^{(q+1)}(s)|^2_r\right)ds\right]\nonumber\\
% &+\sum_{i=1}^M \mathcal{B}(u_0,J_i).
\end{align}
For $n=1,$ combining \eqref{recursion_eta} with \eqref{interpolation_estimate_psi2}, \eqref{interpolation_estimate_psi_int}, \eqref{interpolant_space}, \eqref{rho-interpolant} and 
(\ref{rho_past}) and using that $\eta(0)=0,$ we find that there exists a generic constant $C$, 
independent of the time step $k_1$ and the spatial mesh size $h,$ such that
\begin{align}\label{first_iterate}
\eta(t_1)&\leq C_3\int_{-\tau_{max}}^{t_1}\Vert \rho(s)\Vert^2+C_4\Vert\rho^0\Vert^2+\frac{1}{\epsilon^2(1-\epsilon^2)}\|\rho^{1}\|^2\nonumber\\
&\leq C\left[\sum_{i=1}^M \mathcal{B}(u_0,J_i)+ h^{2r+2}\left(\|u(0)\|^2_{r+1}+\|u(t_1)\|^2_{r+1}\right)\right.\nonumber\\
&\left.+   h^{2r+2}k_1 \|u\|^2_{r+1,I_1} +k_1^{2q+2}\int_{I_1}\Vert \pd{s}{q+1}u(s,\cdot)\Vert^2 ds\right].
\end{align} 
For $n=2,$ using again \eqref{recursion_eta} and then \eqref{interpolant_space}, \eqref{rho-interpolant}, \eqref{general_formula} and \eqref{first_iterate}, we 
find that there is a constant $C$, such that 
\begin{align}
\eta(t_2)&\leq C_1\eta(t_1)+C_2\int_{t_1-\tau_{max}}^{t_1}\eta(s)ds +C_3\int_{t_1-\tau_{max}}^{t_2}\Vert \rho(s)\Vert^2\nonumber\\[5pt]
&+C_4\Vert\rho^1\Vert^2+\frac{1}{\epsilon^2(1-\epsilon^2)}\|\rho^{2}\|^2\nonumber\\
&\leq C\left[\sum_{i=1}^M m(i)\mathcal{B}(u_0,J_i) + h^{2r+2}\left(\|u(0)\|^2_{r+1}+\|u(t_1)\|_{r+1}^2
+\|u(t_2)\|^2_{r+1}\right)\right.\nonumber\\
&\left.+ h^{2r+2}\sum_{j=1}^2 m(j) k_j \|u\|^2_{r+1,I_j} +\sum_{j=1}^2 m(j)k_j^{2q+2}\int_{I_j}
\Vert \pd{s}{q+1}u(s,\cdot)\Vert^2 ds\right],
\end{align}
where $m(i)$ and $m(j)$ are the multiplicity how many times we visited the interval $J_i$ and $I_j,$ respectively, 
in the integral of $\Vert\rho(s)\Vert$ over the delay interval. If $\tau_{max}$ is large compared to the time step, 
then $m(i)$ is consequently also larger.  

We can repeat this procedure for the subsequent time intervals, which completes the proof of the theorem.
\end{proof}
%%%%%%%%%%%%%%%%%%%%%%%%%%%%%%%%%%%%%%%%%%%%%
\section{Numerical simulations}\label{s:simulations}
In this section we present applications of the FEM discretization to the neural field equations, 
starting with delay differential equations with constant delay. 
%%%%%%%%%%%%%%%%%%%%%%%%%%%%%%%%%%%%%
\subsection{DDE with constant delay}
Here we study the numerical solution of equations of the form
\begin{align}\label{test1}
\dot{u}(t)&=f\left(u(t),u(t-\tau)\right),\nonumber\\
u(s)&=u_0(s),\quad s\in[-\tau,0],
\end{align}
with $\tau>0$ a constant delay and $f:\R^2\to \R$ linear, given by 
$f\left(u(t),u(t-\tau)\right)=-\alpha u(t)+u(t-\tau).$ 

To verify  our results on the error analysis, we compare the time-discontinuous Galerkin FEM 
solution (dG(1)) using linear basis functions, with the exact solution computed for some delay intervals. Let the history function be 
$u_0(s)=-s,$ $s\in[-\tau,0
]$ and $\tau=2.$ Figure \ref{fig:constant_delay} illustrates the solution when $\alpha=1,$ 
for which we know that it converges to a non-zero steady 
state. We set $k_n=k$ for all $n$ and distinguish two cases. First, when $\tau/k$ is not 
an integer, then the 
dG(1) method is second order accurate, which is consistent with our result on the error estimate. 
When $\tau/k$ is, however, integer then we observe 
a higher order accuracy of order three. Figure \ref{fig:constant_delay} 
shows both cases.
\begin{figure}%[htbp]%[H]
\centering
\subfigure[]{\epsfig{file = 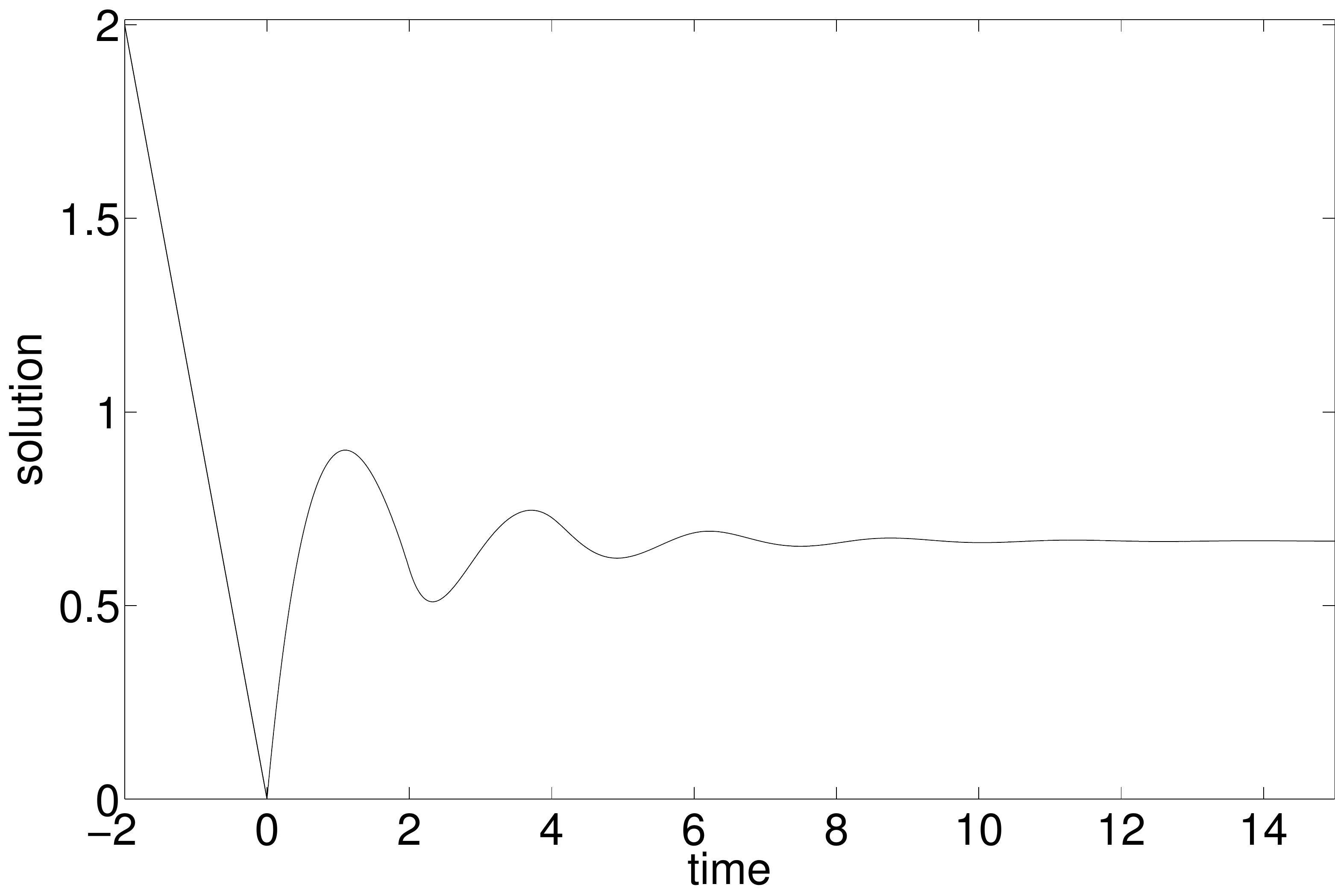, width = 0.456\textwidth}}
\hspace{\fill}
  \subfigure[]{\epsfig{file = 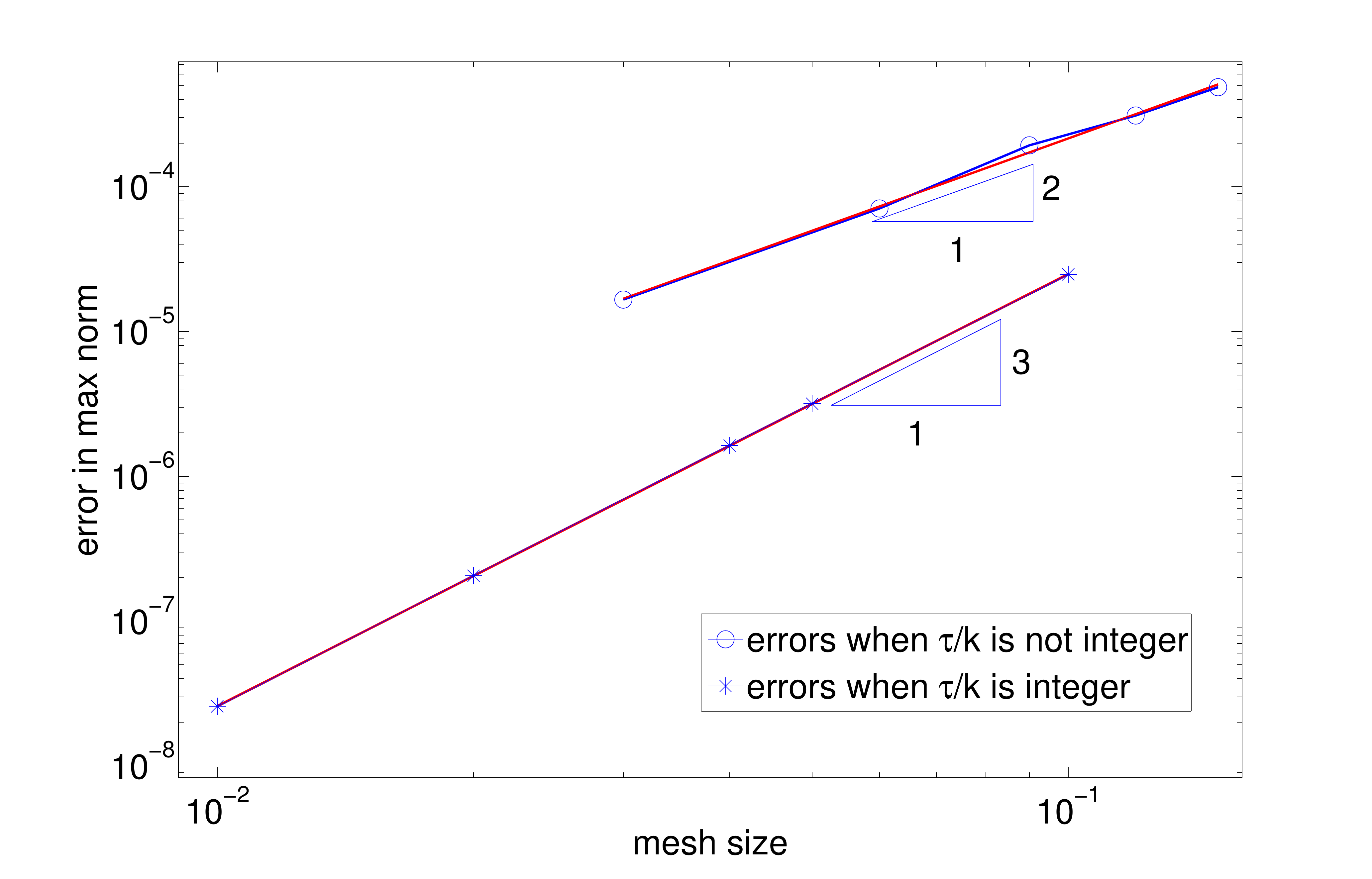, width = 0.52\textwidth}}
\caption{(a) dG(1) solution of (\ref{test1}) for $\alpha=1.$ (b) Discretization error in a log-log plot when $\tau/k$ is not an integer and when it is an integer.} \label{fig:constant_delay}
\end{figure}
 
The numerical integration of delay differential equations is very sensitive to jump discontinuities in the 
solution or in its derivatives. Such discontinuity points are referred in the literature as breaking points, \cite{Bellen}. In
case of constant delay, the breaking points are $\xi_n=n\tau$ for $n=1,2,\dots .$ The best procedure to guarantee the required accuracy is to include these breaking points in the set of mesh points. In our example, the 
derivative of the solution has discontinuity at $t=0.$ When the breaking points are also mesh points, i.e., 
when $\tau/k$ is integer then the error in the discontinuous Galerkin method is of order $O(k^{q+2}),$ which 
is of superconvergent order.
%%%%%%%%%%%%%%%%%%%%%%%%%%%%%%%%%%%%%%%%
\subsection{Integro-differential equations}
One important result is the successful treatment of the fully implicit case, i.e., when the delay is zero. 
Hence, consider the integro-differential equation, obtained by removing the delay term in (\ref{NF-1pop_model}) and adding 
a given, sufficiently smooth, source term $g$
\begin{equation}\label{integro-diff-eq}
\frac{\partial u}{\partial t}(t,x)+\alpha u(t,x)=\int_\Omega J(x,r)S\left(u(t,r)\right)d\, r+g(t,x),
\end{equation}
with initial condition $u(0,x)=u_0(x).$ In our numerical 
simulation, we further simplify this equation by taking $J(x,r)=1,$ and $S\left(u(t,r)\right)=u(t,r)$ linear. 

As a first example, we take $g=0,$ $\alpha=1$ and $u_0(x)=x.$ The exact solution of (\ref{integro-diff-eq}) 
is $u(t,x)=x e^{-t},$ which converges to zero as $t\to\infty$ for every $x\in\Omega.$ The time interval is divided equidistantly with time step $k_n=k.$ 
With this example we want to demonstrate that the time accuracy is not 
destroyed when we add a spatial integral term. The dGcG-FEM solution using linear basis functions, both in space and time,
 and the time accuracy for this example are 
plotted in Figure \ref{fig:integral_eq}. We observe that the error in the dGcG-FEM method is of superconvergent order. 
\begin{figure}%[htbp]
\centering
\subfigure[]{\epsfig{file = 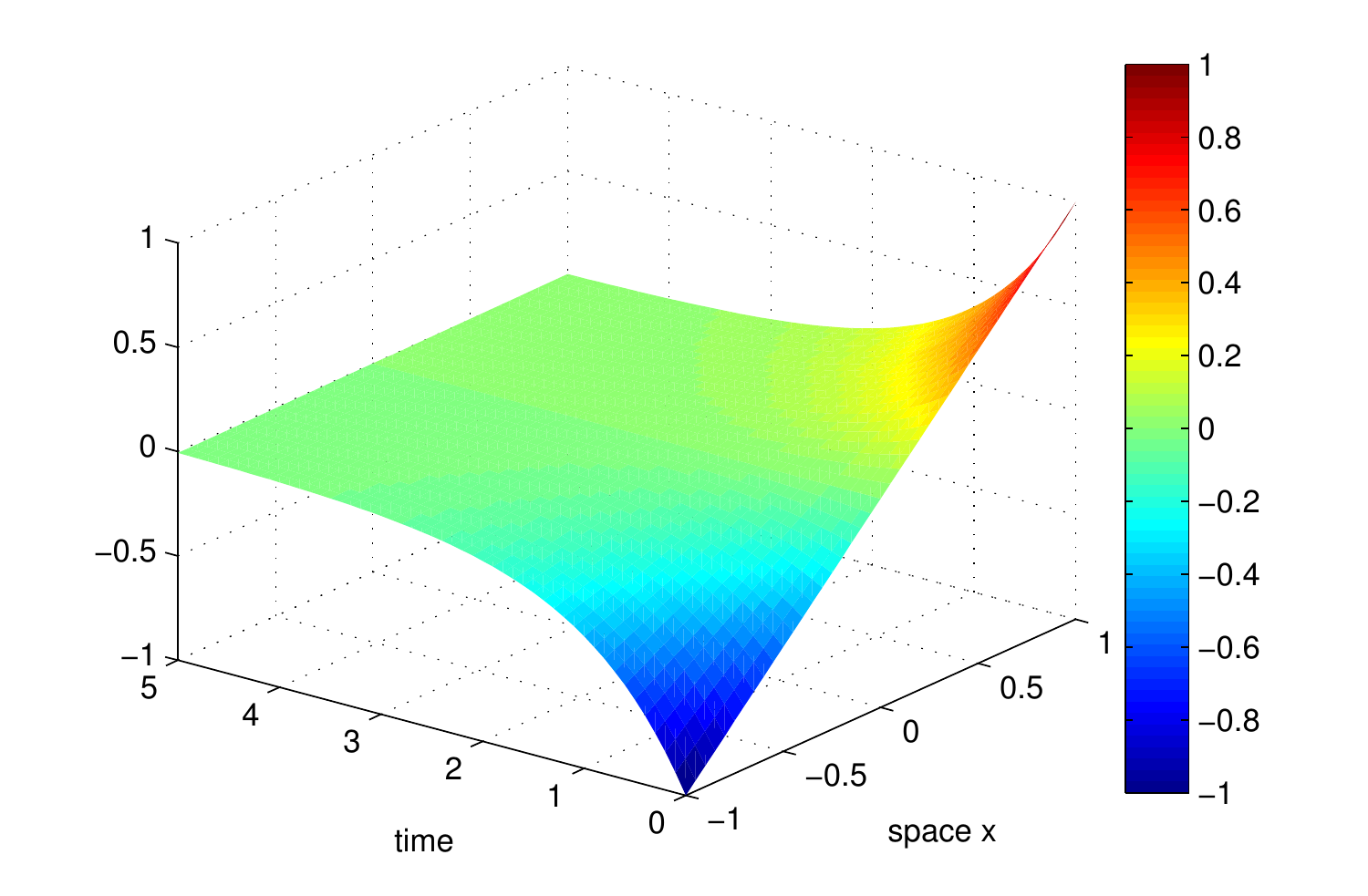, width = 0.51\textwidth}}
\hspace{\fill}
  \subfigure[]{\epsfig{file = 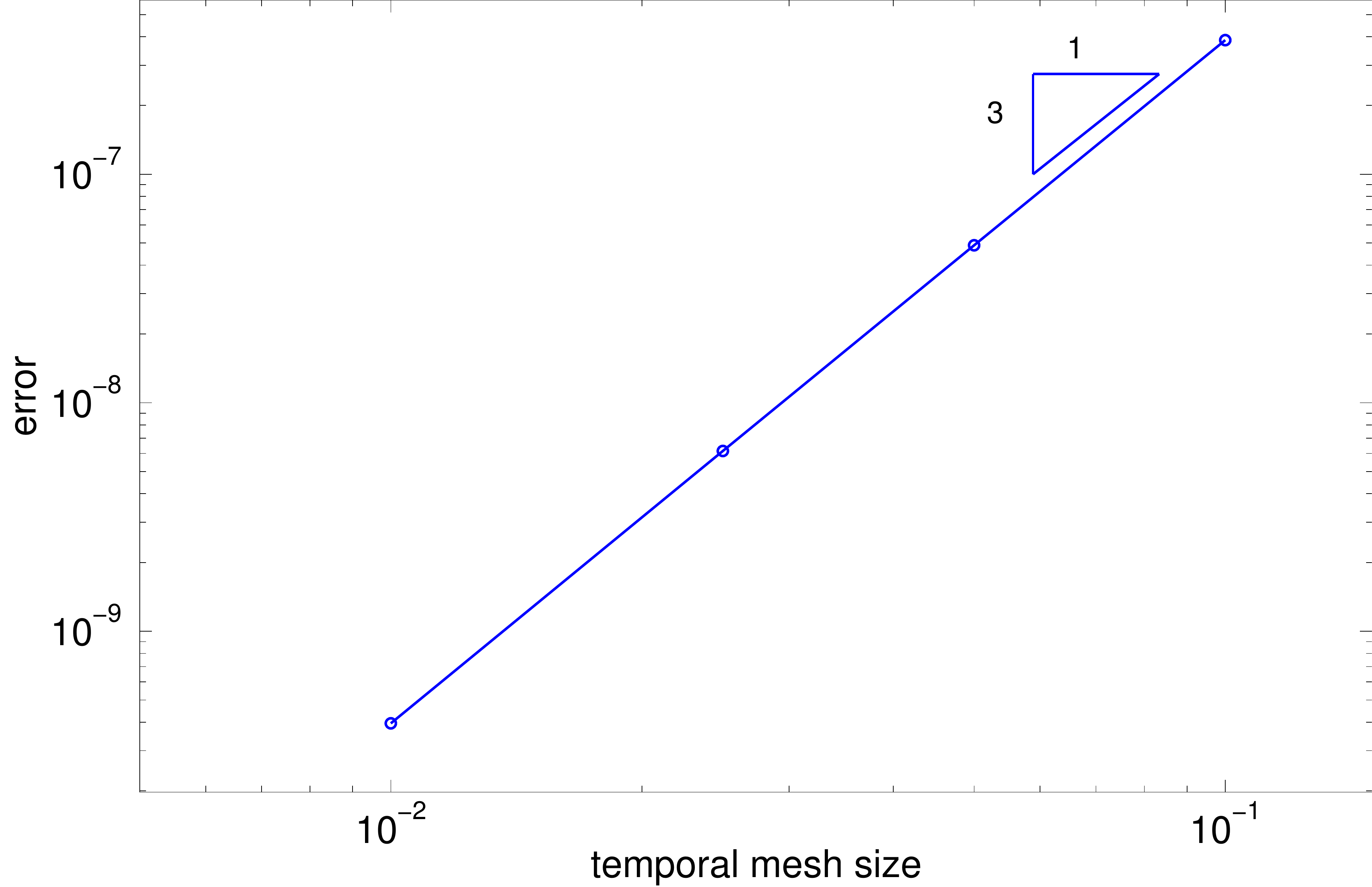, width = 0.47\textwidth}}
\caption{(a) dGcG(1) solution of (\ref{integro-diff-eq}) for $\alpha=1$ and $g=0.$ (b) Discretization error 
in a log-log plot.} \label{fig:integral_eq}
\end{figure}

In the second example, we study the time accuracy when the solution of (\ref{integro-diff-eq}) is periodic in time. 
Take $g(t,x)=x\cos t+\alpha x \sin t,$ $\Omega=[-1,1].$ Then the exact solution $u(t,x)=x\sin t$ satisfies the 
initial condition $u(0,x)=0.$ We compute the error of the solution at several time levels in a 
period and observe the same temporal accuracy as in the first example. 
%%%%%%%%%%%%%%%%%%%%%%%%%%%%%%%%%%%%%%%%%%%%%%%
\subsection{The neural field equations}
In this section we demonstrate the dGcG(1) method for an example analyzed in \cite{vg}, both analytically 
and numerically. Consider the single population model (\ref{NF-1pop_model}), when the space is 1-dimensional. 
Space and time are rescaled such that $\bar\Omega=[-1,1]$ and the propagation speed is 1. 
This yields
\begin{equation}
\tau(x,y)=\tau_0+|x-y|.
\end{equation}
In this case, equation (\ref{NF-1pop_model}) becomes
\begin{equation}\label{NF1d}
\frac{\partial u}{\partial t}(t,x)=-\alpha u(t,x)+\int_{-1}^1 J(x,r)S(u(t-\tau(x,r),r))dr.
\end{equation}
The connectivity and activation functions are, respectively,
\begin{equation}
J(x,r)=\hat J(x-r)=\sum_{j=1}^N \hat{c}_j e^{-\mu_j |x-r|},\quad \hat{c}_j\in\mathbb{R},\ \mu_j\in\mathbb{R},\ x,r\in[-1,1],
\end{equation} 
and 
\begin{equation}\label{activation}
S(u)=\frac{1}{1+e^{-\sigma u}}-\frac{1}{2},\ \forall u\in\mathbb{R}.
\end{equation}
%\subsubsection{Simulations beyond a Hopf bifurcation} 
Hopf bifurcations play an important role in the analysis of 
neural field equations. By choosing the steepness parameter $\sigma$ of the 
activation function as bifurcation parameter, we can simulate, using the dGcG(1) scheme, the space-time 
evolution of the solution beyond a Hopf bifurcation. As in \cite{vg}, we choose the parameters $\alpha=1$ and $\sigma=6$ in the 
activation function (\ref{activation}) and the delay $\tau_0=1.$ In this simulation the connectivity 
function has a bi-exponential form 
\begin{equation}\label{1dconnectivity}
\hat{J}(x)=\hat{c}_1 e^{-\mu_1 |x|}+\hat{c}_2 e^{-\mu_2 |x|},\quad |x|\leq 1,
\end{equation}
with $\hat{c}_1=3.0,\ \hat{c}_2=-5.5,\ \mu_1=0.5,\ \mu_2=1.0.$ 
Figure \ref{f:ContourHopf} shows the time evolution of the system and 
Figure \ref{SurfaceHopf} is a surface plot of the numerical solution.  
\begin{figure}%[htbp]
\centering
\includegraphics[scale=.35]{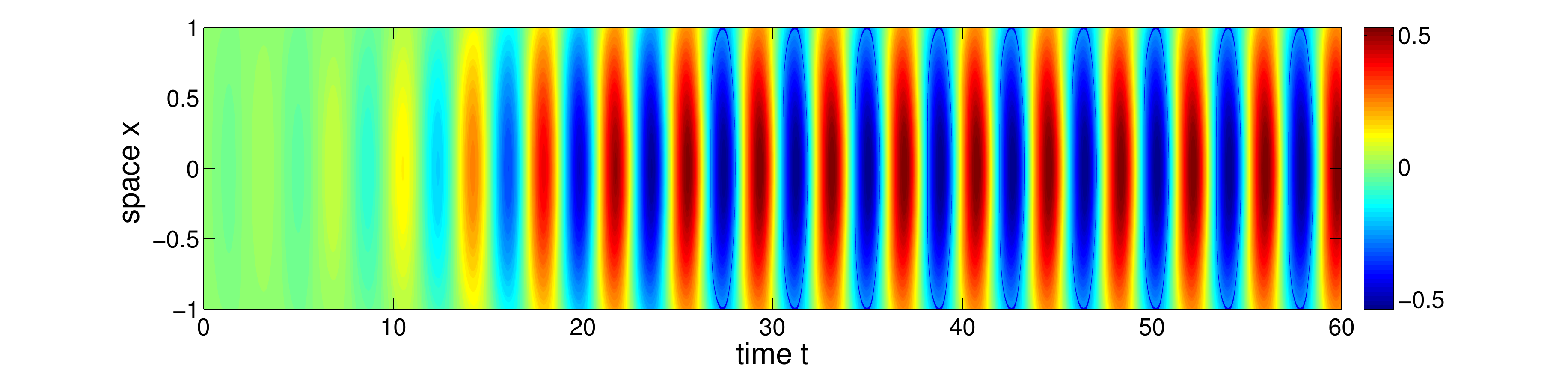}
\caption{Time evolution of system (\ref{NF1d}) for $\sigma=6,$ beyond a Hopf bifurcation.}\label{f:ContourHopf}
\end{figure}
\begin{figure}%[htbp]
\centering
\includegraphics[scale=0.45]{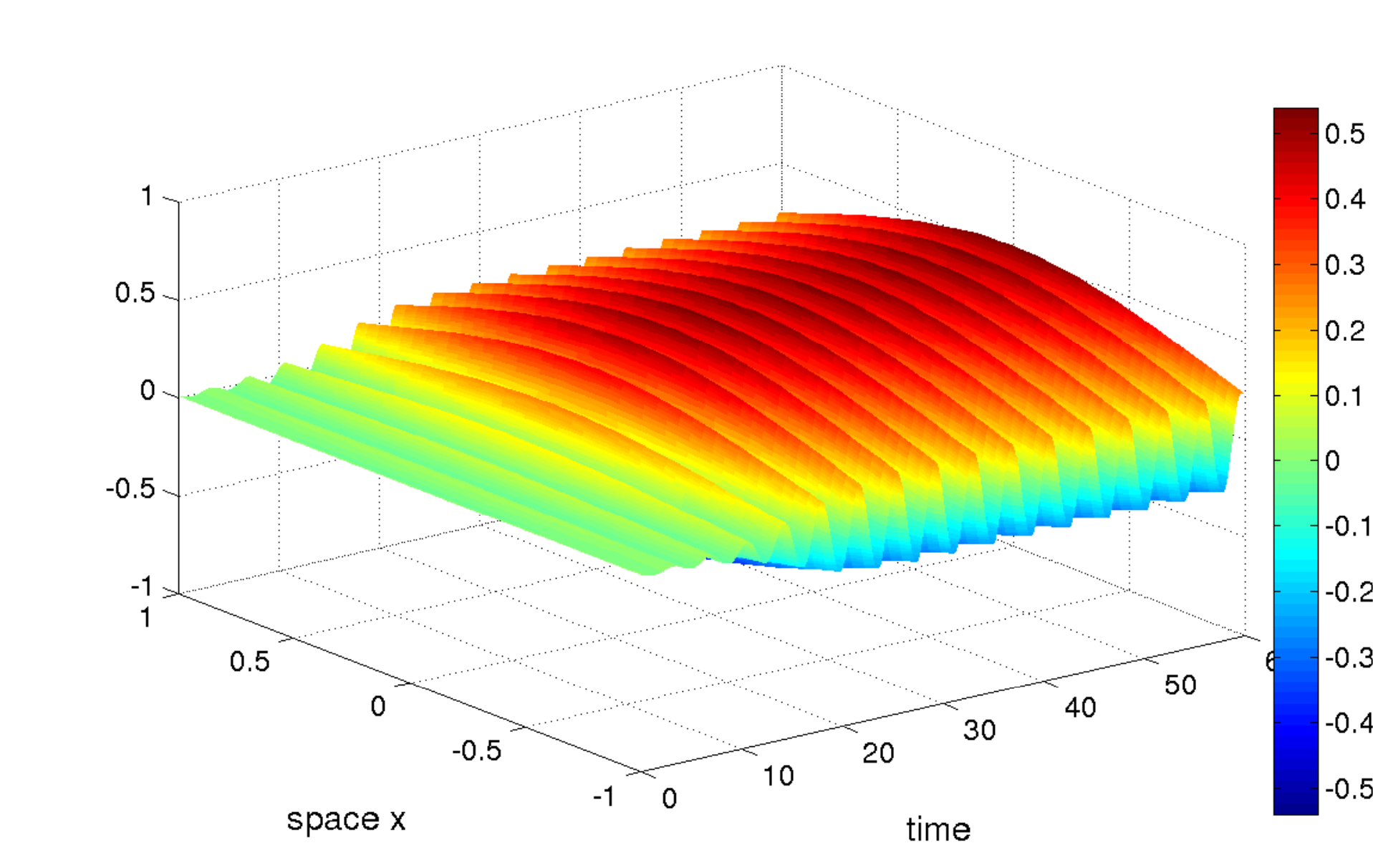}
\caption{Surface plot of the time evolution of the system (\ref{NF1d}) for $\sigma=6,$ beyond a Hopf bifurcation.}\label{SurfaceHopf}
\end{figure}
The initial function for this simulation is $u(t,x)=\epsilon=0.01,$ $t\in[-\tau_{max},0].$ 
Note that, because the size of the delay is relatively large compared to the time step, we do not need to linearize the 
system to solve the algebraic equations with a Newton method. 
%%%%%%%%%%%%%%%%%%%%%%%%%%%%%%%%%%%%%%%%%%%%%%%%%%%
\subsection{Neural fields with spatial inhomogeneity}\label{s:inhomogenity}
Consider the neural field equation \eqref{NF1d} with the locally changed connectivity 
\begin{equation}
\tilde J(x,y)=J(x,y)+\omega J(x,y)\mid_{\tilde\Omega}, \quad \omega>0,
\end{equation}
where $J$ is given in \eqref{1dconnectivity} with the same parameters and $\tilde\Omega\subset\Omega.$ The activation function 
is given in \eqref{activation} with the bifurcation parameter $\sigma=4,$ chosen below the threshold for Hopf bifurcation to 
occur in the homogeneous case, see \cite{vg}. In Figures \ref{f:Contour_inhom}, \ref{f:Surface_inhom} and \ref{Solution_inhom}, 
we compare the solution of the system with homogeneous kernel, with the solution where we have locally 
changed the connectivity, 
specifically in one element, i.e., $\tilde\Omega=K\in\bar{\mathcal{T}}_h.$ Our simulations show that while the solution converges to a 
steady state in the homogeneous case, in the inhomogeneous case the solution becomes periodic ($\omega=15$). This is a new 
phenomenon observed in the one dimensional case. It requires, however, further bifurcation analysis in the two-parameter space $(\sigma,\omega).$
\begin{figure}%[htbp]
\centering
  \includegraphics[scale=0.4]{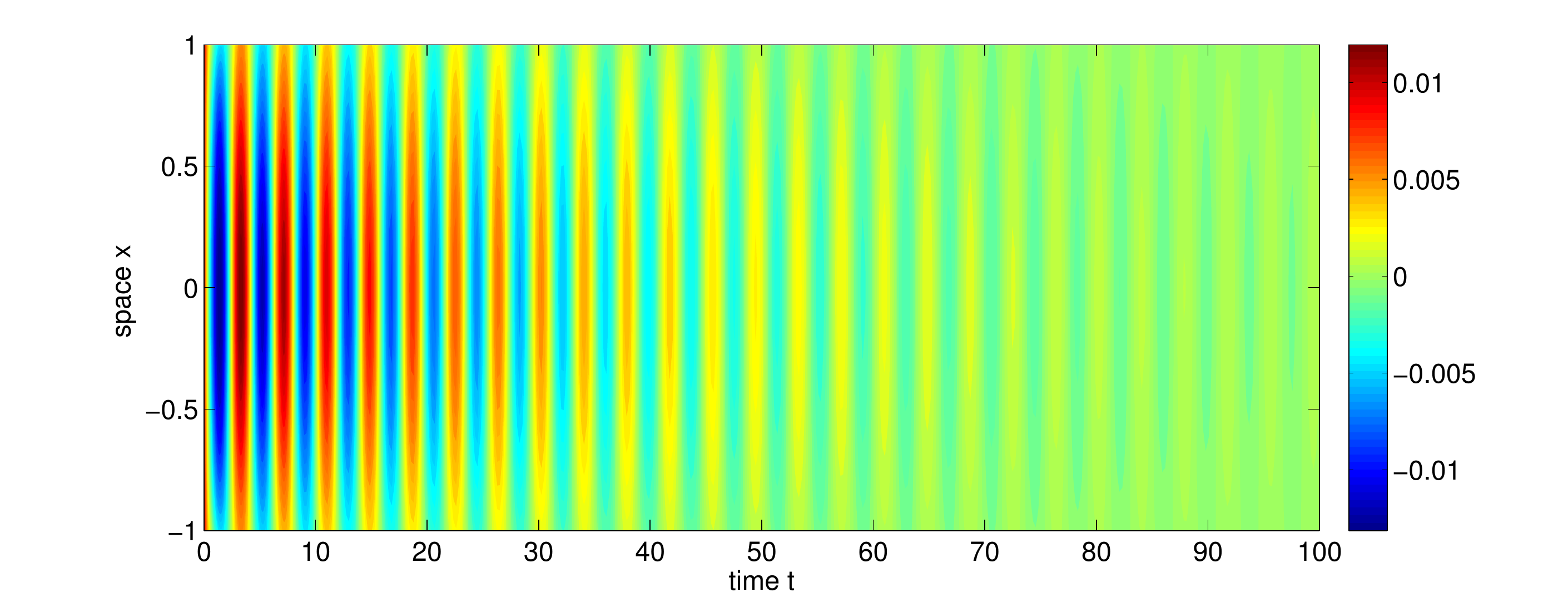}
  \includegraphics[scale=0.39]{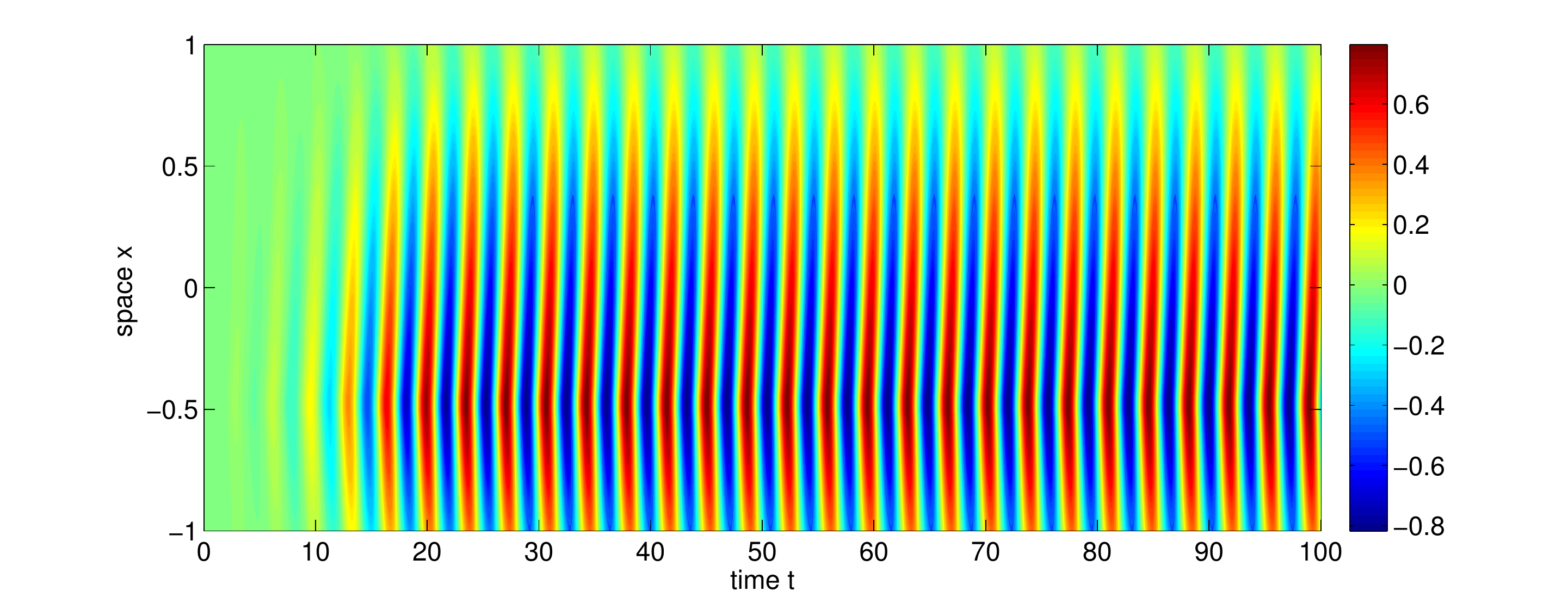}
  \caption{Time evolution of the system (\ref{NF1d}) for $\sigma=4,$ in the homogeneous (top) and the inhomogeneous (bottom) case.}\label{f:Contour_inhom}
\end{figure}
%%%%%%%%%%%%%%%%%%%%%%%%%%%
\begin{figure}%[htbp]
\centering
\subfigure[]{\epsfig{file = 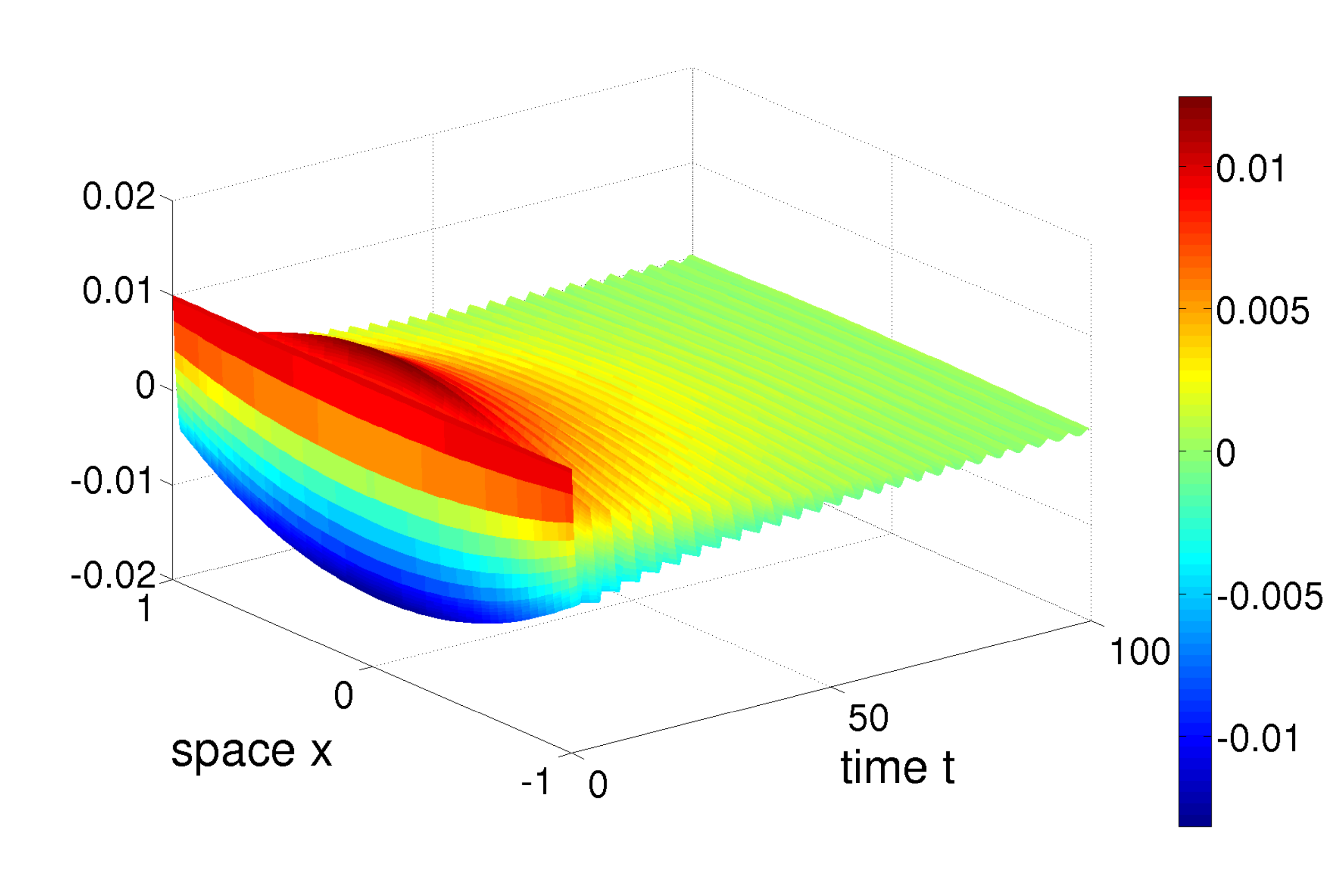, width = 0.48\textwidth}}
%\hspace{\fill}
  \subfigure[]{\epsfig{file = 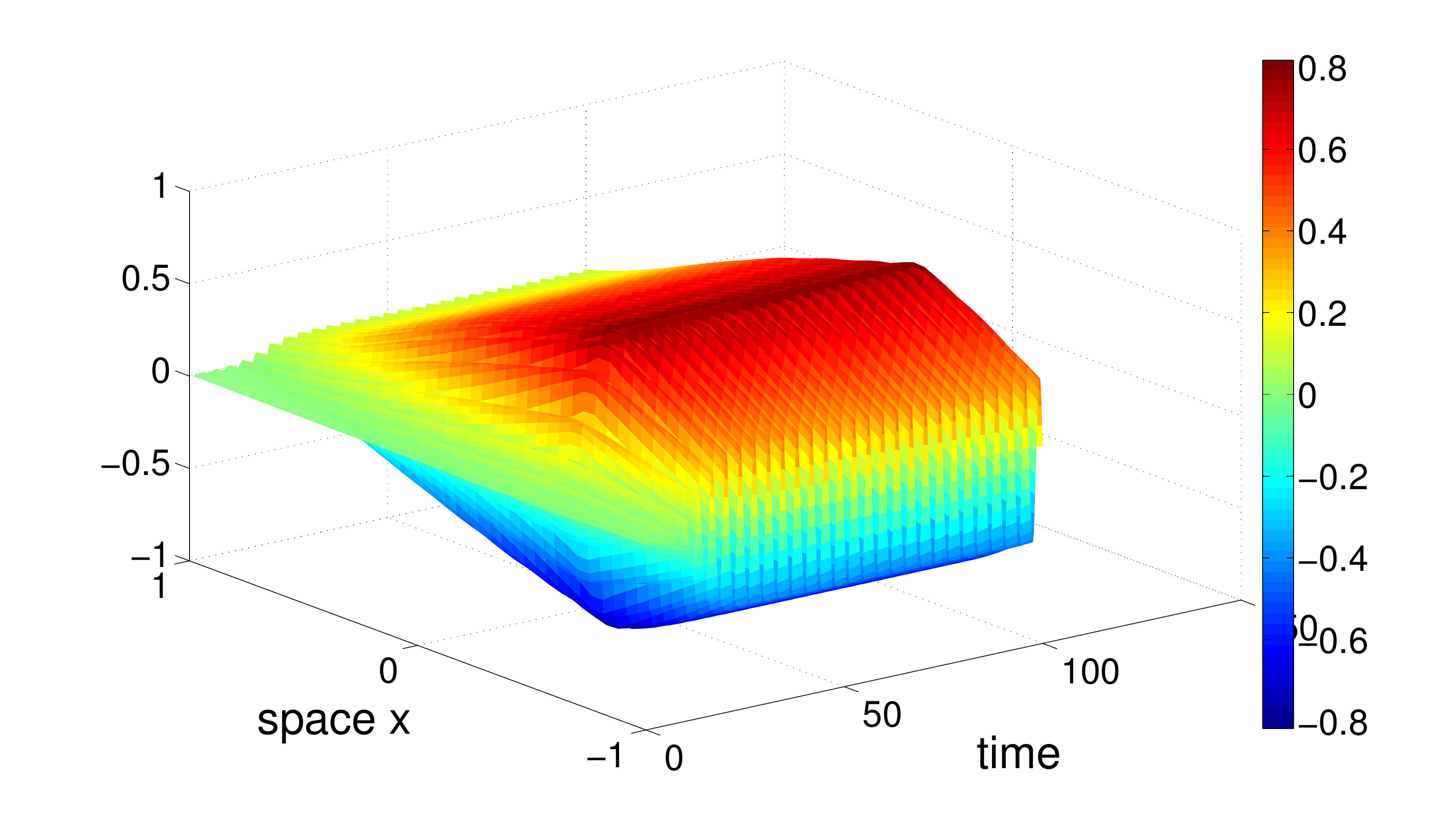, width = 0.5\textwidth}}
  \caption{Time evolution of the system (\ref{NF1d}) at given spatial position $x$ for $\sigma=4,$ in the homogeneous (a) and the inhomogeneous (b) case.}\label{f:Surface_inhom}
\end{figure}
%%%%%%%%%%%%%%%%%%%%%
\begin{figure}%[htbp]
\centering
  \subfigure[]{\epsfig{file = 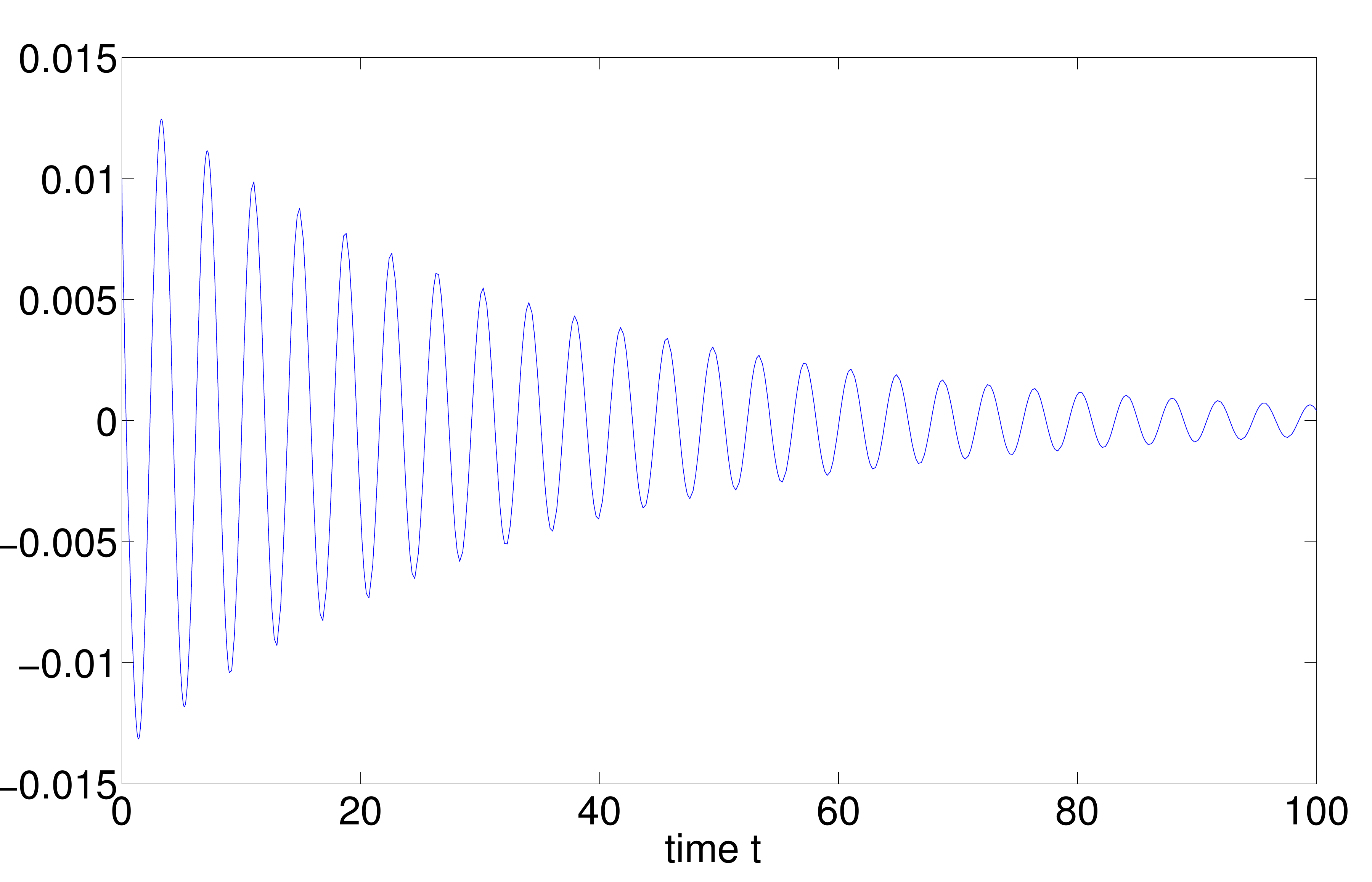, width = 0.47\textwidth}}
%\hspace{\fill}
  \subfigure[]{\epsfig{file = 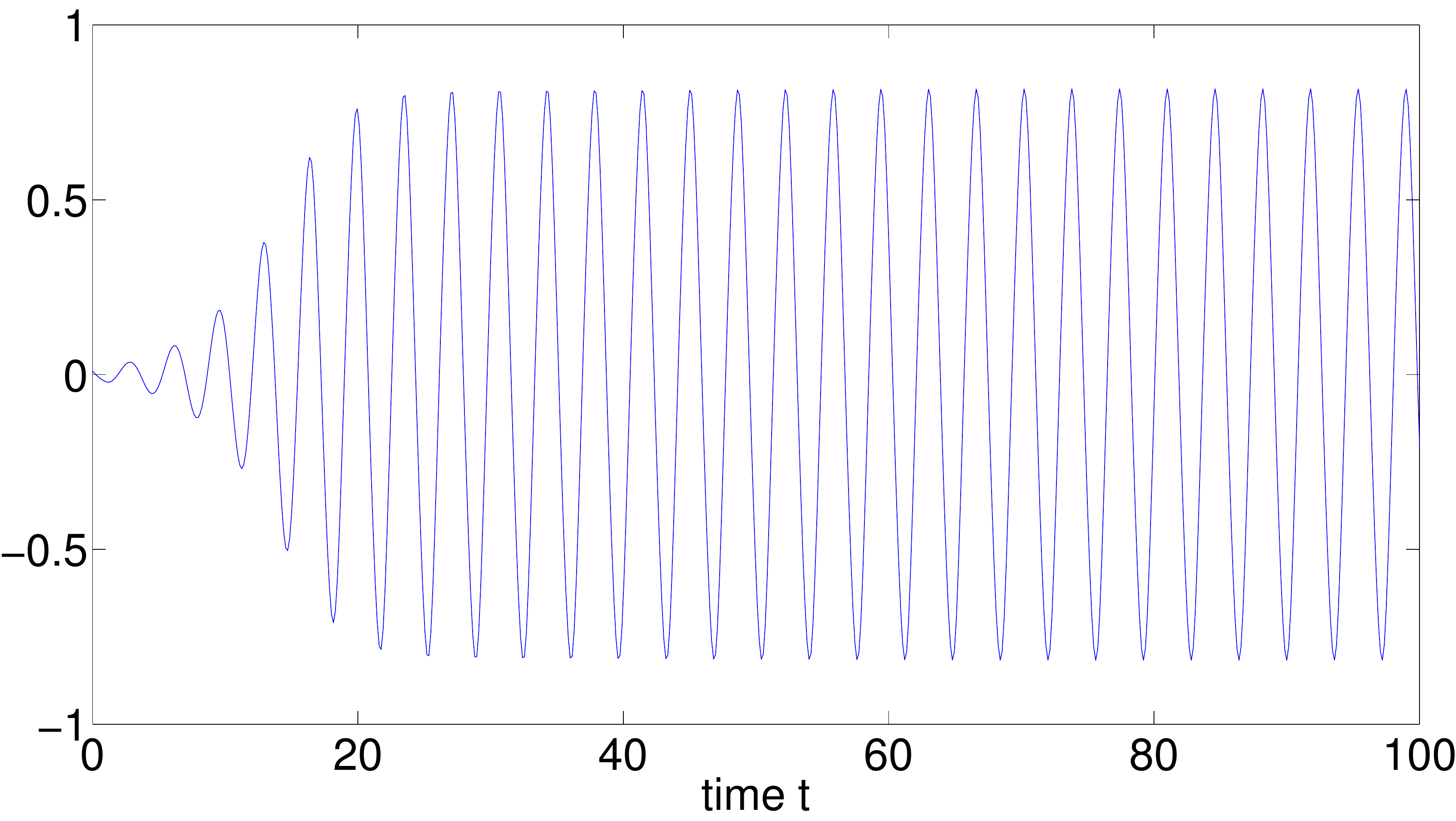, width = 0.51\textwidth}}
  \caption{Time evolution of the system (\ref{NF1d}) at given spatial position $x$ for $\sigma=4,$ in the homogeneous (a) and the inhomogeneous (b) case.}\label{Solution_inhom}
\end{figure}
\section{Concluding remarks}
In this article we have presented a new space-time dGcG-FEM to solve delay integro-differential 
equations with space dependent delays. The main result is an a-priori error estimate of the space-time dGcG 
method, which also shows that the method is numerically stable. We demonstrated that by using a dGcG method we can 
handle general connectivity, synaptic activation and delay functions, and do not need to make any restriction on 
spatial dimension or shape of the 
domain. This makes it possible to extend our model to more general domains as well as more populations in the system, 
which are particularly interesting for our applications. 
%%%%%%%%%%%%%%%%%%%%%%%%%%%%%%%%%%%%%%%%%%%%%%%%%%
%%%%%%%%%%%%%%%%%%%%%%%%%%%
\section*{Acknowledgments}
The first author was supported by the Hungarian Scientific Research Fund, Grant No. K109782. 
The ELI-ALPS project (GOP-1.1.1.-12/B-2012-0001, GINOP-2.3.6-15-2015-00001) is supported by the 
European Union and co-financed by the European Regional Development Fund.
%%%%%%%%%%%%%%%%%%%%%%%%%%%

\end{document}